\documentclass[11pt,twoside]{amsart}
\usepackage{amssymb}
\usepackage{graphicx,color}

\newtheorem{theorem}{Theorem}[section]
\newtheorem{proposition}[theorem]{Proposition}

\newtheorem{lemma}[theorem]{Lemma}
\newtheorem{corollary}[theorem]{Corollary}
\newtheorem{definition}[theorem]{Definition}
\newtheorem{notation}[theorem]{Notation}
\newtheorem{example}[theorem]{Example}
\newtheorem{remark}[theorem]{Remark}

\newcommand{\skipit}[1]{{}}
\newcommand{\prfend}{\hbox to7pt{\hfil}
\par\vskip-\baselineskip\hbox to\hsize
{\hfil\vbox {\hrule width6pt height6pt}}\vskip\baselineskip}

\newcommand{\ZZ}{\mathbb{Z}}

\newcommand{\N}{\mathbb{N}}

\newcommand {\PP}{\mathbb{P}}

\newcommand{\cM}{\mathcal{M}}

\newcommand{\cB}{\mathcal{B}}
\newcommand{\W}{\mathcal{W}_{d}}
\newcommand{\cX}{\mathcal{X}}
\newcommand{\cN}{\mathcal{N}}
\newcommand{\kk}{\mathbb{K}}
\newcommand{\af}{\mathbb{A}}
\newcommand{\Wk}{\mathcal{W}_{d,k}}
\DeclareMathOperator{\reg}{reg}

\DeclareMathOperator{\GL}{GL}
\DeclareMathOperator{\diag}{diag}
\DeclareMathOperator{\GCD}{GCD}
\DeclareMathOperator{\relint}{relint}
\DeclareMathOperator{\I}{I}
\DeclareMathOperator{\syz}{syz}
\DeclareMathOperator{\HS}{HS}
\DeclareMathOperator{\Comh}{H}
\DeclareMathOperator{\comh}{h}
\DeclareMathOperator{\supp}{supp}

\newcommand{\myarrow}[2]{\hbox to #1pt{\hfil$\to$\hfil}{\hskip-#1pt{\raise
10pt\hbox to#1pt{\hfil$\scriptscriptstyle #2$\hfil}}}}

\begin{document}

\title{The canonical module of GT--varieties and the normal bundle of RL--varieties.}

\author[Liena Colarte-Gomez]{Liena Colarte--G\'{o}mez}
\address{Departament de matem\`{a}tiques i Inform\`{a}tica, Universitat de Barcelona, Gran Via de les Corts Catalanes 585, 08007 Barcelona,
Spain}
\email{liena.colarte@ub.edu}

\author[Rosa M. Mir\'o-Roig]{Rosa M. Mir\'o--Roig}
\address{Departament de matem\`{a}tiques i Inform\`{a}tica, Universitat de Barcelona, Gran Via de les Corts Catalanes 585, 08007 Barcelona,
Spain}
\email{miro@ub.edu}

\begin{abstract} In this paper, we study the geometry of $GT-$varieties $X_{d}$ with group a finite cyclic group $\Gamma \subset \GL(n+1,\kk)$ of order $d$.  We prove that the homogeneous ideal $\I(X_{d})$ of $X_{d}$ is generated by binomials of degree at most $3$ and we provide examples reaching this bound. We give a combinatorial description of the canonical module of the homogeneous coordinate ring of $X_{d}$ and we show that it is generated by monomial invariants of $\Gamma$ of degree $d$ and $2d$. This allows us to characterize the Castelnuovo-Mumford regularity of the homogeneous coordinate ring of $X_d$. Finally, we compute the cohomology table of the normal bundle of the so called $RL-$varieties. They are projections of the Veronese variety $\nu_{d}(\PP^{n}) \subset \PP^{\binom{n+d}{n}-1}$ which  naturally arise from level $GT-$varieties.
\end{abstract}

\thanks{The   authors are partially   supported
by  MTM2016--78623-P}

\maketitle

\tableofcontents

\markboth{L. Colarte, R. M. Mir\'o-Roig}{}

%*****************************************************************************
\large

\section{Introduction}
Through this paper, $\kk$ denotes an algebraically  closed field of characteristic zero, $R = \kk[x_0,\hdots, x_n]$ and $\GL(n+1,\kk)$ denotes the group of invertible matrices of size $(n+1) \times (n+1)$ with coefficients in $\kk$. 

In \cite{MM-RO}, Mezzetti, Mir\'o-Roig and Ottaviani related the existence of homogeneous artinian ideals $I \subset R$ generated by homogeneous forms $F_{1},\hdots, F_{r}$ of degree $d$ failing the weak Lefschtez property in degree $d-1$ to the existence of rational projective varieties of $\PP^{\binom{n+d}{n}-r-1}$ satisfying a Laplace equation. They called $I$ a {\em Togliatti system}. Since then Togliatti systems have been extensively studied as one can see in  
\cite{AMRV}, \cite{CMM-R}, \cite{CM-R}, \cite{MM-R1}, \cite{MM-R}, \cite{MkM-R} and \cite{M-RS}.

Any Togliatti system $I$ induces a morphism $\varphi_I: \PP^{n} \to \PP^{r-1}$ defined as $(F_1,\hdots,F_r)$, its image $\varphi_I(\PP^{n})$ is called the {\em variety parameterized by $I$}. In \cite{MM-R}, the authors introduced a new family of Togliatti systems parameterizing varieties with a special geometric property. They called {\em $GT-$system} with cyclic group $\ZZ/d\ZZ$ any Togliatti system  parameterizing a Galois  covering  with group $\ZZ/d\ZZ$. $GT-$systems and the varieties parameterized by them have been subsequently studied in \cite{CMM-RS}, \cite{CM-R} and \cite{CMM-R}, in the latter reference the authors applied invariant theory methods to tackle them. 

To be more precise, fix integers $2 \leq n < d$ and $e$ a $d$th primitive root of $1 \in \kk$. We denote by $M_{d;\alpha_{0},\hdots, \alpha_{n}}$ the diagonal matrix $\diag(e^{\alpha_{0}},\hdots, e^{\alpha_{n}})$, where $0 \leq \alpha_0 \leq \cdots \leq \alpha_n < d$ are integers such that $\GCD(\alpha_0,\hdots,\alpha_n,d) = 1$ and $\alpha_i < \alpha_j$ for some $i \neq j$. We set $\Gamma := \langle M_{d;\alpha_{0},\hdots, \alpha_{n}} \rangle \subset \GL(n+1,\kk)$ a finite cyclic group of order $d$ and $R^{\Gamma} = \{p \in R \,\mid\, g(f) = f, \; \forall g \in \Gamma\}$ the ring of invariants of $\Gamma$. The cyclic extension $\overline{\Gamma} \subset \GL(n+1,\kk)$ of $\Gamma$ is the finite abelian group of order $d^{2}$ generated by $M_{d;\alpha_{0},\hdots, \alpha_{n}}$ and $M_{d;1,\hdots,1} = \diag(e,\hdots,e)$. The ring of invariants of $\overline{\Gamma}$ is $R^{\overline{\Gamma}} = \{p \in R^{\Gamma} \,\mid\, \deg(p) = td, \, 0 \leq t\}$, often called the $d$th Veronese subalgebra of $R^{\Gamma}$. In \cite{CMM-R} it is proved that $R^{\overline{\Gamma}} = \kk[m_1,\hdots,m_{\mu_{d}}]$, where $m_1,\hdots,m_{\mu_d}$ are all the monomial invariants of $\Gamma$ of degree $d$, and it is shown that the ideal $I_{d} = (m_1,\hdots,m_{\mu_{d}})$ is a $GT-$system. They called $GT-$variety with group $\Gamma$ the variety $X_{d}:= \varphi_{I_{d}}(\PP^{n})$ parameterized by $I_d$ and they started to investigate its geometry. 
In \cite{Groebner}, Gr\"obner posed the problem of determining whether a monomial projection of the Veronese variety $\nu_{d}(\PP^{n}) \subset \PP^{\binom{n+d}{n}-1}$ parameterized by the set $\mathcal{M}_{n,d} \subset R$ of all monomials of degree $d$ is an arithmetically Cohen Macaulay (shortly aCM) variety. Motivated by this long-standing problem, the authors of \cite{CMM-R} proved that any $GT-$variety $X_{d}$ is an aCM variety showing that the homogeneous coordinate ring of $X_d$ is isomorphic to $R^{\overline{\Gamma}}$. They also tackled the problem of finding a minimal free resolution of the homogeneous ideal $\I(X_{d})$ of $X_{d}$, which they determined for all $GT-$surfaces. Recently in \cite{CMM-RS}, the notions of $GT-$system and $GT-$variety have been generalized to any finite group acting on $R$, non necessarily cyclic or even abelian. 

In this paper, we address three topics regarding $GT-$varieties. The first two questions concern explicitly the geometry of any $GT-$variety $X_d$, in contrast with the last one, which  deals with the cohomology of the normal bundle of some smooth rational varieties naturally arising from $GT-$varieties. First, we find a set of generators of the homogeneous ideal $\I(X_{d})$ of $X_{d}$.  We prove that $\I(X_{d})$ is generated by homogeneous binomials of degree at most $3$ (Theorem \ref{Theorem: main theorem generators ideal}). We exhibit examples of homogeneous ideals of $GT-$varieties reaching this bound, which also show that it depends on the group $\Gamma$. Second, we determine the algebraic structure of the canonical module $\omega_{X_{d}}$ of the  homogeneous coordinate ring of $X_{d}$. We identify $\omega_{X_{d}}$ with the ideal $\relint(I_{d}) = (x_0^{a_0}\cdots x_n^{a_n} \in R^{\overline{\Gamma}} \,\mid\, 0 \neq a_0\cdots a_n)$ of $R^{\overline{\Gamma}}$ generated by the relative interior of the semigroup ring $\kk[m_1,\hdots,m_{\mu_{d}}]$ (Proposition 4.1). We prove that $\relint(I_d)$ is generated by monomials of degree $d$ and $2d$ (Theorem \ref{Theorem: Canonical Module Gt-variety}). This connection allows us to compute the Castelnuovo-Mumford regularity of $R^{\overline{\Gamma}}$.

Finally, we introduce a new family of smooth rational monomial projections of the Veronese variety $\nu_{d}(\PP^{n}) \subset \PP^{\binom{n+d}{n}-1}$ naturally associated to level $GT-$varieties. A $GT-$variety $X_d$ is called {\em level} if the ideal $\relint(I_d)$ is generated only by monomials of degree $d$ and, hence, $R^{\overline{\Gamma}}$ is a level ring. An {\em $RL-$variety} $\cX_{d}$ associated to a level $GT-$variety $X_d$ is a monomial projection of the Veronese variety $\nu_{d}(\PP^{n}) \subset \PP^{\binom{n+d}{n}-1}$ parameterized by the set of monomials $\mathcal{M}_{n,d} \setminus \{m = x_0^{a_0}\cdots x_n^{a_n} \in \relint(I_{d}) \,\mid\, \deg(m) = d\}$. The name $RL-$variety is conceived to emphasize the relation with the {\bf r}elative interior and the {\bf l}evelness. We give examples of $RL-$varieties in any dimension. Inspired by the recent work of Alzati and Re (\cite{Alzati-Re}), we contribute to the classical open problem of computing the cohomology of the normal bundle of  smooth rational varieties (Theorem \ref{Theorem: dimension cohomology normal}).  Most results and examples of this topic focus on smooth rational curves and surfaces, see for instance \cite{Alzati-Re1}, \cite{Eisenbud-Van de Ven} and \cite{Sacchiero}. We determine the cohomology table of the normal bundle of any $RL-$variety, shading new light on higher dimensions.

\vspace{0.3cm} Let us see how this work is organized. In Section \ref{Section: preliminaries}, we gather the basic definitions and results needed in the body of this paper. Section \ref{Section: homogeneous ideal} is entirely devoted to find a set of homogeneous binomial generators of the homogeneous ideal $\I(X_{d})$ of any $GT-$variety $X_{d}$. We establish that the homogeneous coordinate ring of $X_{d}$ is isomorphic to $R^{\overline{\Gamma}}$ and that $\I(X_{d})$ is a homogeneous prime binomial ideal. Our main result (Theorem \ref{Theorem: main theorem generators ideal}) proves that $\I(X_{d})$ is generated by binomials of degree at most $3$. We give families of examples of $GT-$varieties whose homogeneous ideals are minimally generated by binomials of degree $2$ and $3$. In Section \ref{Section: canonical module}, we investigate the algebraic structure of the canonical module $\omega_{X_{d}}$ of the homogeneous coordinate ring of $X_{d}$. We identify $\omega_{X_{d}}$ with the ideal $\relint(I_{d})$ of $R^{\overline{\Gamma}}$, which gives us a combinatorial description of $\omega_{X_{d}}$. In Theorem \ref{Theorem: Canonical Module Gt-variety} we show that $\relint(I_{d})$ is generated by monomials of degree $d$ and $2d$. We further study $GT-$varieties $X_{d}$ where  $R^{\overline{\Gamma}}$ is a level ring, i.e. the canonical module $\relint(I_d)$ of $R^{\overline{\Gamma}}$ is generated in only one degree. In particular, we study varieties whose homogeneous coordinate ring $R^{\overline{\Gamma}}$ is level and $\relint(I_{d})$ is minimally generated in degree $d$. Afterwards in Theorem \ref{Theorem: regularity}, we characterize the Castelnuovo-Mumford regularity of $R^{\overline{\Gamma}}$. 

Finally in Section \ref{Section: Cohomology of normal bundles of RL-varieties}, we introduce the notions of a level $GT-$variety and its associated $RL-$variety and we give examples of any dimension. Using the new methods of \cite{Alzati-Re}, we compute the cohomology table of the normal bundle of any $RL-$variety (see Theorem \ref{Theorem: dimension cohomology normal}). 

\vskip 4mm \noindent
 {\bf Acknowledgements.}  The authors are grateful to the anonymous referee for providing detailed comments which have improved the exposition of this paper.
\section{Preliminaries}
\label{Section: preliminaries}

In this section, we introduce the main objects and results we use in the body of this paper. First, we define semigroups and normal semigroups, we relate them to invariant theory of finite groups
and we see a geometrical interpretation of these objects. For more details the reader can look at \cite{Hochster}, \cite{Bruns-Herzog} and \cite{Sturmfelds}. Finally, we define the weak Lefschetz property, we recall the notions of $GT-$systems and $GT-$varieties and we collect some basic results on this topic. 

\vspace{0.3cm}
\noindent{\bf Semigroup rings and rings of invariants.}
\label{semigroup rings and rings of Invariant}
By a {\em semigroup} we mean a finitely generated additive subsemigroup $H = \langle h_{1}, \hdots, h_{t} \rangle \subset \ZZ_{\geq 0}^{n+1}$. $L(H)$ is the additive subgroup of $\ZZ^{n+1}$ generated by $H$. We denote by $\kk[H] \subseteq R$ the semigroup ring associated to $H$, i.e. the graded $\kk-$algebra generated by the monomials $X^{h_{j}} = x_0^{a_0^j}\cdots x_n^{a_n^{j}} \in R$ associated to the points $h_{j} = (a_0^{j},\hdots,a_n^{j}) \in H$, $j = 1,\hdots,t$.

\begin{definition} \rm  \label{Definition: Normal semigroup} A semigroup $H \subset \ZZ^{n+1}$ is called {\em normal} if it satisfies the following condition: if $zh \in H$ for some $h \in L(H)$ and $0 \neq z \in \ZZ_{\geq 0}$, then $h \in H$. 
\end{definition}

%%%%%%%%%%%%%%%%%%%%%%%%%%%%%%%%%%%%%%%%%%%%%%%%%%%%%%%%%
\iffalse
Concerning normal semigroups, Hochster proves the following result.

\begin{proposition}\label{Proposition:Hochster CM} If a semigroup $H$ is normal, then $\kk[H]$ is Cohen-Macaulay.
\end{proposition}
\begin{proof} See \cite[Theorem 1]{Hochster}.
\end{proof}
\fi
%%%%%%%%%%%%%%%%%%%%%%%%%%%%%%%%%%%%%%%%%%%%%%%%%%%%%%%%%

A large family of normal semigroups comes from invariant theory, precisely those associated to  finite abelian groups acting on $R$. To be more precise, let $\Lambda \simeq \ZZ/d_1\ZZ  \oplus \cdots \oplus \ZZ/d_r\ZZ$ and choose $d_{i}$-th primitive roots $e_{i}$ of $1\in \kk$ , $i = 1, \hdots, r$. Therefore $\Lambda$ can be linearly represented in $\GL(n+1,\kk)$ by means of $r$ diagonal matrices $\diag(e_{i}^{u_{0,i}}, \hdots, e_{i}^{u_{n,i}})$, where $u_{j,i} \in \ZZ_{\geq 0}$, $0 \leq j \leq n$, $1 \leq i \leq r$. Let $R^{\Lambda} := \{p \in R \,\mid\, \lambda(p) = p \; \text{for all} \; \lambda \in \Lambda\}$ be the ring of invariants of $\Lambda$ acting on $R$. Since $\Lambda$ acts diagonally, each monomial $x_{0}^{a_{0}} \cdots x_{n}^{a_{n}} \in R$ is mapped into a multiple of itself by every $\lambda \in \Lambda$, and a polynomial $p \in R^{\Lambda}$ if and only if all its monomials are invariants of $\Lambda$. Thus, by the Noether's degree bound (see \cite[Theorem 2.1.4]{Sturmfelds}), $R^{\Lambda}$ has a finite basis consisting of monomials of degree at most the order of $\Lambda$. By a basis of $R^{\Lambda}$ we mean a set of elements $\{\theta_1,\hdots, \theta_l\} \subset R^{\Lambda}$ which minimally generates $R^{\Lambda}$ as a $\kk-$algebra, i.e. $R^{\Lambda} = \kk[\theta_1,\hdots,\theta_{l}]$. 
Let $X^{h_{1}}, \hdots, X^{h_{t}}$ be a monomial basis of $R^{\Lambda}$ and $H = \langle h_{1}, \hdots, h_{t} \rangle$. Then $R^{\Lambda} = \kk[H]$.
Furthermore, a monomial $x_{0}^{a_{0}} \cdots x_{n}^{a_{n}} \in R^{\Lambda}$ if and only if $(a_{0}, \hdots, a_{n})$ satisfies the system of congruences:
\begin{equation}
a_{0}u_{0,i} + \cdots + a_{n}u_{n,i} \equiv 0 \pmod{d_i},\  i = 1,\hdots, r.
\end{equation}
Now, if $w \in L(H)$ is such that $zw \in H$ for some $z \in \ZZ_{\geq 0}$, then  $w \in H$, so $H$ is normal. By \cite[Theorem 1]{Hochster} or \cite[Proposition 13]{Hochster-Eagon}, $\kk[H]$ is Cohen Macaulay. 

%%%%%%%%%%%%%%%%%%%%%%%%%%%%%%%%%%%%%%%%%%%%%%%%%%%%%%%%

%%%%%%%%%%%%%%%%%%%%%%%%%%%%%%%%%%%%%%%%%%%%%%%%%%%%%%%%%%%%%
More generally, let $G \subset \GL(n+1,\kk)$ be a finite group. Geometrically, the ring  $R^{G}$ of invariants can be regarded as the coordinate ring of the quotient of $\mathbb A^{n+1}$ by $G$. To be more precise, set $\{f_{1},\hdots, f_{t}\}$ a basis of $R^{G}$, often called a set of {\em fundamental invariants of $G$},  and let $\kk[w_{1},\hdots, w_{t}]$ be  the polynomial ring in the new variables $w_{1},\hdots, w_{t}$. Then the quotient of $\mathbb A^{n+1}$ by $G$ is given by the morphism $\pi: \mathbb A^{n+1}\to \pi(\mathbb A^{n+1})\subset \mathbb A^t$, such that $\, \pi(a_{0},\hdots, a_{n})= (f_{1}(a_{0},\hdots, a_{n}), \hdots, f_{t}(a_{0},\hdots, a_{n}))$. Even further, $\pi$ is a Galois covering of $\pi(\af^{n+1})$ with group $G$. For further details on quotients varieties we refer the reader to \cite{Serre}. The ideal $\I(\pi(\mathbb A^{n+1}))$ of the quotient variety is called the {\em ideal of syzygies} among the invariants $f_{1},\hdots,
f_{t}$; it is the kernel of the homomorphism defined by $w_{i} \to f_{i}$, $i = 1,\hdots, t$. We denote it by $\syz(f_{1},\hdots, f_{t})$. We summarize all these facts in the following proposition. 

\begin{proposition}\label{Proposition: qv by fg acting linearly on poly ring} Let $G \subset \GL(n+1,\kk)$ be a finite linear group, $\{f_{1},\hdots, f_{t}\}$ be a set of fundamental invariants and let $\pi: \af^{n+1} \to \af^{t}$ be the induced morphism. Then,
\begin{itemize}
\item [(i)] $\pi(\af^{n+1})$ is the quotient of $\mathbb{A}^{n+1}$ by $G$ with affine coordinate ring $R^{G}$.
\item [(ii)] $R^{G} \cong \kk[w_{1},\hdots, w_{t}]/\syz(f_{1},\hdots, f_{t}).$

\item [(iii)] $\pi$ is a Galois covering of $\pi(\af^{n+1})$ with group $G$. 
\end{itemize}
\end{proposition}
\begin{proof} See \cite[Section 6]{Stanley}.
\end{proof}

The cardinality of a general orbit $G(a)$, $a \in \af^{n+1}$, is called the degree of the covering. Moreover, if we can find a homogeneous set of fundamental invariants $\{f_1,\hdots, f_t\}$ of $G$ such that $\pi: \PP^{n} \to \PP^{t-1}$ is a morphism, then  the projective version of Proposition \ref{Proposition: qv by fg acting linearly on poly ring} is true.

\vspace{0.3cm}
\noindent{\bf GT--systems and GT--varieties.}
\label{Lefschetz properties and Togliatti systems}
Let $I\subset R $ be a homogeneous artinian ideal. We say that $I$ has the \emph{weak Lefschetz property}  (WLP)
if there is a linear form $L \in R_1$ such that, for all
integers $j$, the multiplication map
\[
\times L: (R/I)_{j-1} \to (R/I)_j
\]
has maximal rank. In \cite{MM-RO}, Mezzetti, Mir\'{o}-Roig and Ottaviani proved
that the failure of the WLP is related to the existence of varieties satisfying at least one Laplace
equation of order greater than 2. More precisely, they proved:

\begin{theorem} \label{tea} Let $I\subset R$ be an artinian
ideal
generated
by $r$ forms $F_1,\dotsc,F_{r}$ of degree $d$ and let $I^{-1}$ be its Macaulay inverse system.
If
$r\le \binom{n+d-1}{n-1}$, then
  the following conditions are equivalent:
\begin{itemize}
\item[(i)] $I$ fails the WLP in degree $d-1$;
\item[(ii)]  $F_1,\dotsc,F_{r}$ become
$\kk-$linearly dependent on a general hyperplane $H$ of $\PP^n$;
\item[(iii)] the $n$-dimensional   variety
 $Y=\overline{\varphi(\PP^{n})}$,
where
$\varphi \colon\PP^n \dashrightarrow \PP^{\binom{n+d}{n}-r-1}$ is the rational morphism associated to $(I^{-1})_d$,
  satisfies at least one Laplace equation of order
$d-1$.
\end{itemize}
\end{theorem}
\begin{proof} See \cite[Theorem 3.2]{MM-RO}.
\end{proof}

Motivated by  the above results,   Mezzetti, Mir\'{o}-Roig and Ottaviani  introduced the following definitions (see \cite{MM-RO} and \cite{MM-R1}):
\begin{definition} \rm Let $I \subset R$ be an artinian ideal generated by $r \leq \binom{n+d-1}{n-1}$ forms  of degree $d$. We say that:

\begin{itemize}
\item[(i)] $I$ is a \emph{Togliatti system} if it fails the WLP in degree $d-1$.
\item[(ii)] $I$ is a \emph{monomial Togliatti system} if, in addition, $I$ can be generated
by monomials.
\end{itemize}
\end{definition}

In particular, a Togliatti system is called {\em smooth} if the variety $Y$ in Theorem \ref{tea}(iii) is smooth. The name is in honour of Togliatti who proved that for
$n = 2$ the only smooth
Togliatti system of cubics is \[I = (x_0^3,x_1^3,x_2^3,x_0x_1x_2)\subset \kk[x_0,x_1,x_2]\]
(\cite{T1} and \cite{T2}). The systematic study of Togliatti systems was initiated in \cite{MM-RO}
and it has been continuing in \cite{MM-R},  \cite{MM-R1}, \cite{AMRV}, \cite{M-RS} and \cite{MkM-R}. Precisely in \cite{MM-R1}, it was introduced the notion of \emph{GT-system} with group a finite cyclic group.  Recently in \cite{CMM-RS1}, this notion has been generalized as follows. 

\begin{definition} \rm \label{Defi:GT}
  A \emph{GT-system} with a finite group $G$ is an artinian ideal $I_{d}\subset R$ generated by $r$ forms $F_{1},\dotsc,F_{r}$ of degree $d$
  such that:
\begin{enumerate}
\item[(i)] $I_{d}$ is a Togliatti system.
\item[(ii)] The morphism $\varphi_{I_{d}}\colon\PP^{n}\rightarrow \PP^{r-1}$ defined by $(F_{1},\dotsc,F_{r})$ is a Galois covering with group $G$.
\end{enumerate}
If conditions (i) and (ii) holds, we say that $\varphi_{I_{d}}(\PP^{n})$ is a {\em $GT-$variety with group $G$.} 
\end{definition}

$GT-$systems with group a finite cyclic group have been extensively studied in \cite{MM-R}, \cite{CMM-RS} and \cite{CMM-R}, while in \cite{CMM-RS1}, the authors investigate $GT-$systems with the dihedral group acting on $\kk[x_0,x_1,x_2]$. In the last two references, invariant theory techniques have been applied to tackle both objects. Fix $G \subset \GL(n+1,\kk)$ a finite group of order $d$. Assume that the ring $R^{G}$ has a basis $\cB$ formed by homogeneous invariants of $G$ of degree $d$ and set $I_{d}$ the ideal generated by $\cB$. Keeping this notation, we have the following. 
\begin{proposition} If $|\cB| \leq \binom{d+n-1}{n-1}$, then $I_{d}$ is a $GT-$system with group $G$. 
\end{proposition}
\begin{proof} Since $I_{d}$ contains a homogeneous system of parameters of $R^{G}$, it is an artinian ideal. By Proposition  \ref{Proposition: qv by fg acting linearly on poly ring}, the associated morphism $\varphi_{I_{d}}$ is a Galois covering with group $G$. By Theorem \ref{tea}, it is enough to prove that $I_{d}$ fails the WLP in degree $d-1$, i.e. for any linear form $L \in R_1$, the multiplication map $\times L: (R/I_d)_{d-1} \to (R/I_d)_d$ is not injective. Let $L \in R_1$ and consider $F := \prod_{Id \neq g \in G} g(L) \in R_{d-1}$. We have that $L\cdot F = \prod_{g \in G}g(L) \in R^{G}$ and, hence, $F \in \ker(\times L)$.
\end{proof}

Let us see an illustrative example.     
\begin{example} \rm (i) Fix $n = 2$, $d = 5$ and $e$ a $5$th primitive root of $1$. The finite cyclic group $\ZZ/5\ZZ$ can be linearly represented by $\Gamma = \langle \diag(1,e,e^{3}) \rangle \subset \GL(3,\kk)$. All the monomial invariants of $\Gamma$ of degree $5$ are: $x_{0}^5,x_{0}^{2}x_{1}^{2}x_{2}, x_{0}x_{1}x_{2}^{3}, x_{1}^{5}, x_{2}^{5}$ (see \cite[Example 2.15]{CMM-R}). In total we have $r = 5$ monomials. The inequality $r\leq \binom{n+d-1}{n-1}$ is satisfied and the ideal $I_5 \subset R$ generated by them fails the WLP in degree $4$. The morphism  $\varphi_{I_5}: \PP^{2} \to \PP^{4}$ is a Galois covering of degree $5$ with group $\Gamma$ (see \cite[Corollary 3.4]{CMM-R} and \cite[Theorem 3.4]{MM-R}). Actually, $\varphi_{I_5}(\PP^{2})$ is the quotient surface of $\PP^{2}$ by $\Gamma$. 

(ii) Fix $n = 3, d = 4$ and $e$ a $4$th primitive root of $1$. The diagonal matrix $\diag(1,e,e^2,e^3)$ generates a finite cyclic subgroup $\Gamma \subset \GL(4,\kk)$ of order $4$. There are exactly $r = 10$ monomial invariants of $\Gamma$ of degree $4$: 
$x_0^4, x_0^2x_1x_3, x_0^2x_2^2, x_0x_1^2x_2, x_0x_2x_3^2, x_1^4, x_1^2x_3^2, x_1x_2^2x_3, x_2^4, x_3^4$ (see \cite[Example 3.2]{CM-R}). The inequality $r \leq \binom{n+d-1}{n-1}$ is satisfied and the ideal $I_4$ generated by them fails the WLP in degree $3$. The associated morphism $\varphi_{I_{4}}: \PP^{3} \to \PP^{9}$ is a Galois covering with group $\Gamma$ (see \cite[Proposition 3.3]{CM-R} and \cite[Corollary 3.4]{CMM-R}). Similarly to (i), $\varphi_{I_{4}}(\PP^{3})$ is the quotient threefold of $\PP^{3}$ by $\Gamma$. 
\end{example} 

In \cite{CMM-R}, \cite{CMM-RS1}, and previously in \cite{CM-R}, 
the authors focus on the geometry of $GT-$varieties with group a finite cyclic group or a dihedral group. In \cite[Theorem 3.2]{CMM-R}, it is proved that all $GT-$varieties with group a finite cyclic group are aCM varieties and in \cite[Proposition 4.3]{CMM-RS1}, it is shown its analogous for $GT-$surfaces with a dihedral group. Furthermore, in \cite[Theorem 4.14]{CMM-R} and \cite[Theorem 4.6]{CMM-RS1}, it is determined a minimal free resolution of $GT-$surfaces with group a finite cyclic group and a dihedral group, respectively. Reference \cite{CM-R} is devoted to find a minimal set of homogeneous binomial generators of the homogeneous ideal of certain $GT-$threefolds with group a finite cyclic group. Our goal is to extend the results obtained so far for $GT-$surfaces and $GT-$threefolds to arbitrary $n$-dimensional $GT-$varieties with group a finite cyclic group.

\section{On the homogeneous ideal of GT--varieties}
\label{Section: homogeneous ideal} 
 
In this section, we look for a system of generators of the homogeneous ideal of $GT-$varieties with group a finite cyclic group. Using combinatorial techniques, we prove that all these ideals are generated by homogeneous binomials of degree $2$ and $3$. We begin introducing some notations. 

\begin{notation}\label{Notation: group cyclic} \rm Fix integers $2 \leq n < d$ and $e$ a $d$th primitive root of $1 \in \kk$. We denote by $M_{d;\alpha_0,\hdots, \alpha_{n}}$ the diagonal matrix $\diag(e^{\alpha_0}, \hdots, e^{\alpha_{n}})$, where $0 \leq \alpha_{0} \leq \cdots \leq \alpha_{n} < d$ are integers such that $\GCD(d,\alpha_0,\hdots, \alpha_n) = 1$ and $\alpha_i < \alpha_j$ for some $i \neq j$.
\end{notation}

Along this section, we fix a finite cyclic group $\Gamma := \langle M_{d;\alpha_0,\hdots, \alpha_n} \rangle \subset \GL(n+1,\kk)$ of order $d$ and we consider the ring $R^{\Gamma}$ of invariants of $\Gamma$ endowed with the natural grading $R^{\Gamma} = \bigoplus_{t \geq 1} R^{\Gamma}_{t}$, \; $R^{\Gamma}_{t} := R_t \cap R^{\Gamma}$. The cyclic extension $\overline{\Gamma} \subset \GL(n+1,\kk)$ of $\Gamma$ is the finite abelian group of order $d^2$ generated by $M_{d;\alpha_{0},\hdots, \alpha_{n}}$ and $M_{d;1,\hdots,1}$. We consider the ring $R^{\overline{\Gamma}}$ of invariants of $\overline{\Gamma}$ with the grading $R^{\overline{\Gamma}} = \bigoplus_{t \geq 1} R^{\overline{\Gamma}}_{t}$, \, $R^{\overline{\Gamma}}_{t} := R^{\Gamma}_{td} = R_{td} \cap R^{\Gamma}$, often called the $d$th Veronese subalgebra of $R^{\Gamma}$. We denote by $\cM_{d} := \{m_{1},\hdots, m_{\mu_{d}}\} \subset R$ the set of all monomial invariants of $\Gamma$ of degree $d$, ordered lexicographically.  We denote $I_{d} \subset R$ the monomial artinian ideal generated by $\cM_{d}$. By $\varphi_{I_{d}}: \PP^{n} \to \PP^{\mu_{d}-1}$, we denote the associated morphism and we set $X_{d} := \varphi_{I_{d}}(\PP^{n}) \subset \PP^{\mu_{d}-1}$ its image. In \cite{CMM-R}, it is established the following. 

\begin{theorem}\label{Theorem: Previos results CMM-R} (i) $\cM_{d}$ is a basis of $R^{\overline{\Gamma}}$.

(ii) $R^{\overline{\Gamma}}$ is the coordinate ring of $X_{d}$. Hence, $X_{d}$ is an aCM monomial projection of the Veronese variety $\nu_{d}(\PP^{n}) \subset \PP^{\binom{n+d}{n}-1}$ from the inverse system $I_{d}^{-1}$. 

(iii) If $\mu_{d} \leq \binom{d+n-1}{n-1}$, then $I_{d}$ is a $GT-$system with group $\Gamma$. In this case, we call $X_{d}$ a $GT-$variety with group $\Gamma$. 
\end{theorem}
\begin{proof}
See \cite[Theorem 3.1-3.3 and Corollary 3.4]{CMM-R}. 
\end{proof}

\begin{remark}\rm Through the rest of this paper, if the condition $\mu_d \leq \binom{d+n-1}{n-1}$ is satisfied, we refer to $I_{d}$ as a $GT-$system and to $X_{d}$ as a $GT-$variety. 
\end{remark}

\begin{example} \rm (i) Fix integers $2 \leq n$ and fix 
$\Gamma = \langle M_{n+1;0,1,2,\hdots,n} \rangle \subset \GL(n+1,\kk)$ (see \cite[Example 3.6(iii)]{CMM-R}). In \cite[Theorem 4.8]{CMM-RS}, it is proved that the condition $\mu_{n+1} \leq \binom{2n}{n-1}$ is satisfied, thus $I_d$ is a monomial Togliatti system. By Theorem \ref{Theorem: Previos results CMM-R}(iii), $I_{d}$ is a $GT-$system.

(ii) Fix integers $2 = n < d$ and fix $\Gamma = \langle M_{d;0,\alpha_1,\alpha_{2}} \rangle \subset \GL(3,\kk)$ (see \cite{MM-R} and \cite[Example 3.6(i)]{CMM-R}). In \cite[Theorem 3.4]{MM-R}, it is proved that the condition $\mu_d \leq \binom{n-1+d}{n-1}$ is satisfied, thus $I_d$ is a Togliatti system. By Theorem \ref{Theorem: Previos results CMM-R}(iii), $I_{d}$ is a $GT-$system.  
\end{example}

Let $w_{1},\hdots, w_{\mu_{d}}$ be new variables and set $S := \kk[w_{1},\hdots, w_{\mu_{d}}]$. We denote by $\I(X_{d}) \subset S$ the homogeneous ideal of $X_{d}$. By Proposition \ref{Proposition: qv by fg acting linearly on poly ring}, $\I(X_{d})$ is the kernel of the morphism $\rho: S \to R$ given by $\rho(w_{i}) = m_{i}$. It holds that $\I(X_{d})$ is the homogeneous binomial prime ideal generated by 
\[\W = \{w_{i_{1}}\cdots w_{i_{k}} - w_{j_{1}}\cdots w_{j_{k}} \in S \; \text{such that} \;  m_{i_{1}}\cdots m_{i_{k}} = m_{j_{1}}\cdots m_{j_{k}}, \; k \geq 2\}.\]
For any $k \geq 2$, we denote by $\Wk$ the set of all binomials of $\W$ of degree exactly $k$. 
Our goal is to prove that $\I(X_{d})$ is generated by binomials of degree $2$ and $3$, i.e. the ideal $I(X_d) = \langle \mathcal{W}_{d,2}, \mathcal{W}_{d,3} \rangle$. Through families of examples in \ref{Example: 3fold homo ideal and canonical module} we observe that this bound is sharp. We start with some definitions.  

\begin{definition}\rm 
Fixed $k \geq 2$, we define a {\em suitable $k-$binomial} to be a non-zero  binomial $w^{\alpha} = w^{\alpha_{+}} - w^{\alpha_{-}} := \prod_{l=1}^{k}
w_{i_{l}} - \prod_{l=1}^{k} w_{j_{l}} \in \Wk$, i.e. $\prod_{l=1}^{k}m_{i_l} = \prod_{l=1}^{k}m_{j_l}$.
\end{definition}

\begin{definition}\rm
Given a suitable $k-$binomial $w^{\alpha} = w^{\alpha_{+}}-w^{\alpha_{-}} \in \Wk$, we denote by $\supp_{+}(w^{\alpha})$ (respectively $\supp_{-}(w^{\alpha})$) the support of the monomial $w^{\alpha_{+}}$ (respectively
support of $w^{\alpha_{-}}$). We say that $w^{\alpha}$ is {\em non trivial} if $\supp_{+}(w^{\alpha}) \cap \supp_{-}(w^{\alpha}) = \emptyset$. Otherwise, we say that $w^{\alpha}$ is {\em trivial}.
\end{definition}

\begin{definition}\rm Let $w^{\alpha} = w^{\alpha_{+}}- w^{\alpha_{-}}\in \Wk$ be a non trivial suitable $k-$binomial. By an $\I(X_{d})_{k}-${\em sequence} from $w^{\alpha_{+}}$ to $w^{\alpha_{-}}$ we mean a finite sequence $(w^{1}, \hdots, w^{t})$ of monomials of $S$ satisfying the following two conditions:
\begin{itemize}
\item[(i)] $w^{1} = w^{\alpha_{+}}, w^{t} = w^{\alpha_{-}}$ and
\item[(ii)] For all $1\le j<t$,  $w^{j}-w^{j+1}$ is  a trivial suitable $k-$binomial.
\end{itemize}
\end{definition}

\begin{example}\label{Example: 3fold homo ideal and canonical module} \rm
Let $\Gamma = \langle M_{6;0,1,2,3} \rangle \subset \GL(4,\kk)$ be a finite group of order $6$. There are $\mu_{6} = 16$ monomial invariants of $\Gamma$, we have:

\[\begin{array}{lll}
\cM_{6} & = &\{x_0^6, x_0^4x_3^2, x_0^3x_1x_2x_3, x_0^3x_2^3, x_0^2x_1^3x_3, x_0^2x_1^2x_2^2, x_0^2x_3^4, x_0x_1^4x_2, x_0x_1x_2x_3^3, x_0x_2^3x_3^2, x_1^6, x_1^3x_3^3, \\
& &  x_1^2x_2^2x_3^2, x_1x_2^4x_3, x_2^6, x_3^6\}, 
\end{array}\]
By Theorem \ref{Theorem: Previos results CMM-R}(iii), the ideal $I_6$ generated by them is a $GT-$system and its associated variety $X_6$ is a $GT-$variety. The homogeneous binomials $w_1w_{15} - w_4^2$ and $w_3w_{12}w_{15} - w_6w_9w_{14}$ are non trivial suitable binomials of degree $2$ and $3$, respectively.  
Indeed, $\rho(w_1w_{15}) = \rho(w_1)\rho(w_{15}) = (x_0^6)(x_2^6) = (x_0^3x_2^3)^2 = \rho(w_4)^2$
and on the other hand $\rho(w_3)\rho(w_{12})\rho(w_{15}) = (x_0^3x_1x_2x_3)(x_1^3x_3^3)(x_2^6) = (x_0^2x_1^2x_2^2)(x_0x_1x_2x_3^2)(x_1x_2^4x_3) = \rho(w_6)\rho(w_9)(\rho(w_{14}))$. Finally,  $\{w_3w_{12}w_{15}, $ $w_5w_9w_{15},w_6w_9w_{14}\}$ is an $\I(X_6)_3-$sequence from $w_3w_{12}w_{15}$ to $
w_{6}w_9 w_{14}$. 
 
\end{example}

The following result characterizes the inclusion of ideals $\langle \Wk \rangle \subset \langle \mathcal{W}_{d,k-1} \rangle$, $k \geq 3$. It is key to prove our main results, which we state later. 

\begin{proposition}\label{Proposition: Criterion sequence} Fix $k \geq 3$ and let $w^{\alpha} = w^{\alpha_{+}}-w^{\alpha_{-}} \in \Wk$ be a suitable $k-$binomial. Then $w^{\alpha} \in \langle \mathcal{W}_{d,k-1} \rangle$ if and only if there exists an $\I(X_{d})_{k}-$sequence from $w^{\alpha_{+}}$ to $w^{\alpha_{-}}$.
\end{proposition}
\begin{proof} We apply the same arguments as in \cite[Proposition 5.4]{CM-R}.
\end{proof}

We need the following lemma: 
\begin{lemma}\label{Lemma: EGZ} Any sequence of $2d-1$ integers in $\{0,\hdots,d-1\}$ contains some subsequence of $d$ integers the sum of which is a multiple of $d$.
\end{lemma} 
\begin{proof}
See \cite[Theorem]{EGZ}.
\end{proof}

\begin{theorem}\label{Theorem: main theorem generators ideal} Let $\Gamma = \langle M_{d;\alpha_0,\hdots,\alpha_n} \rangle \subset \GL(n+1,\kk)$ be a cyclic group of order $d$ and $X_d \subset \PP^{\mu_{d}-1}$ the variety parameterized by the ideal $I_d = (m_1,\hdots, m_{\mu_{d}})$ generated by all monomial invariants of $\Gamma$ of degree $d$. Then, the homogeneous ideal $\I(X_{d})$ of $X_d$ is generated by quadrics and cubics. Precisely, $\I(X_d) = \langle \mathcal{W}_{d,2}, \mathcal{W}_{d,3} \rangle$. 
\end{theorem}

\begin{proof} First, we prove that for all $k \geq 4$, any non trivial suitable $k-$binomial admits an $\I(X_{d})_{k}-$sequence. By Proposition \ref{Proposition: Criterion sequence}, this implies $\langle \Wk \rangle \subset \langle \mathcal{W}_{d,k-1}\rangle$. Fix $k \geq 4$ and let $w^{\alpha} = w^{\alpha_{+}}-w^{\alpha_{-}} = w_{i_{1}}\cdots w_{i_{k}} - w_{j_{1}}\cdots w_{j_{k}}$ be a non trivial suitable $k-$binomial.  For each $w_{i_{l}}$ (respectively $w_{j_{l}}$), let $m_{i_{l}} = x_{0}^{a_{0}^{l}}\cdots x_{n}^{a_{n}^{l}}$ be its associated monomial (respectively $m_{j_{l}} = x_{0}^{b_{0}^{l}}\cdots x_{n}^{b_{n}^{l}}$), $l = 1,\hdots,k$. We have that 
\begin{equation}\label{Equation:1} 
\sum_{l=1}^{k} a_{s}^{l} = \sum_{l=1}^{k}
b_{s}^{l}, \quad  0 \leq s \leq n.
\end{equation} 
We consider the monomials $m_{i_1}, m_{j_1}$ and for each $0 \leq s \leq n$ we define:
\[c_{s} = \left\{\begin{array}{lll}
0 & \quad &\text{if} \;\; a_{s}^{1} > b_{s}^{1}\\
b_{s}^{1}-a_{s}^{1} & \quad & \text{otherwise.} 
\end{array}\right.\]
This gives rise to a non-zero monomial $m = x_{0}^{c_{0}}\cdots x_{n}^{c_{n}} \in R$ of degree strictly smaller than $d$. Clearly, $m$ divides $m_{i_{2}}\cdots m_{i_{k}}$ (see (\ref{Equation:1})). Thus we consider $m' = (m_{i_{2}}\cdots m_{i_{k}})/m$, which is a monomial of degree at least $(k-2)d \geq 2d$. 
We write $m' = x_{0}^{f_{0}}\cdots x_{n}^{f_{n}}$. We define the sequence of integers $L = (\alpha_0, \overset{f_0}{\hdots} ,\alpha_{0},\hdots, \alpha_{n}, \overset{f_{n}}{\hdots},\alpha_{n})$. Since $L$ has length at least $(k-2)d \geq 2d$, by Lemma \ref{Lemma: EGZ}, we can find a subsequence $L' 
= (\alpha_0 ,\overset{g_0}{\hdots}, \alpha_{0},\hdots, \alpha_{n}, \overset{g_{n}}{\hdots},\alpha_{n}) \subset L$ of $d$ elements whose sum is a multiple of $d$, which implies that the monomial $x_{0}^{g_{0}}\cdots x_{n}^{g_{n}} \in R^{\overline{\Gamma}}$. By Theorem \ref{Theorem: Previos results CMM-R}(i), we can decompose 
\[m_{i_{2}}\cdots m_{i_{k}} = m_{i_{2}}^{1} \cdots m_{i_{k}}^{1},\]
where all $m_{i_{l}}^{1} \in R^{\overline{\Gamma}}$, $2 \leq l \leq k$, are monomials of degree $d$ and in particular:
\[m_{i_{k}}^{1} = x_{0}^{g_{0}}\cdots x_{n}^{g_{n}}.\]
Notice that we have $m_{i_{1}}\cdots m_{i_{k}} = m_{i_{1}}m_{i_{2}}^{1} \cdots m_{i_{k}}^{1}$. 
We define $w^{2} \in S$ to be the monomial $\rho^{-1}(m_{i_{1}})\rho^{-1}(m_{i_{2}}^{1}) \cdots \rho^{-1}(m_{i_{k}}^{1})$. By construction, $w^{\alpha_{+}}-w^2$ is a  trivial suitable $k-$binomial.
Observe that $m = x_{0}^{c_{0}}\cdots x_{n}^{c_{n}}$  divides $m_{i_{2}}^{1}\cdots m_{i_{k-1}}^{1}$, thus $m_{j_{1}}$ divides $m_{i_{1}}m_{i_{2}}^{1}\cdots m_{i_{k-1}}^{1}$. Applying the same argument as before, we factorize
\[m_{i_{1}}m_{i_{2}}^{1}\cdots m_{i_{k-1}}^{1} = m_{i_{1}}^{2}m_{i_{2}}^{2}\cdots m_{i_{k-1}}^{2},\]
where $m_{i_{1}}^{2} = m_{j_{1}}$ and all $m_{i_{l}}^{2} \in R^{\overline{\Gamma}}$, $2 \leq l \leq k-1$, are monomials of degree $d$. We set $w^{3} = \rho^{-1}(m_{i_{1}}^{2})\cdots \rho^{-1}(m_{i_{k-1}}^{2})\cdot \rho^{-1}(m_{i_{k}}^{1})$. Since $m_{i_{1}}m_{i_{2}}^{1}\cdots m_{i_{k-1}}^{1}m_{i_{k}}^{1} = m_{i_{1}}^{2}m_{i_{2}}^{2}\cdots m_{i_{k-1}}^{2}m_{i_{k}}^{1}$, 
$w^{2}-w^{3}$ is a suitable trivial $k-$binomial. Furthermore, since 
$m_{i_{1}}^{2} = m_{j_{1}}$, also $w^{3}-w^{\alpha_{-}}$ is a trivial suitable $k-$binomial. Therefore $(w_{i_{1}}\cdots w_{i_{n}},w^{2},w^{3},w_{j_{1}}\cdots w_{j_{n}})$ is an $\I(X_{d})_{k}-$sequence, from which it follows that $\langle \Wk \rangle \subset \langle \mathcal{W}_{d,k-1} \rangle$. 
The argument we have developed only requires that $(k-2)d \geq 2d$, which it is satisfied for all $k \geq 4$. Thus we have proved that for all $k \geq 4$, 
\[\langle \Wk \rangle \subset \langle \mathcal{W}_{d,k-1} \rangle \subset \cdots \subset  \langle \mathcal{W}_{d,3} \rangle,\]
which completes the proof.
\end{proof}

Despite $\mathcal{W}_{d,2}$ always belongs to a minimal set of homogeneous binomial generators of $\I(X_{d})$, it is not the case for $\mathcal{W}_{d,3}$. Even further, this fact depends on the action of the group $\Gamma = \langle M_{d;\alpha_{0},\hdots, \alpha_{n}} \rangle$, as we illustrate in the following examples. 

\begin{example} \label{Example: Cubics in the ideal} \rm  (i) Fix integers $2 = n < d$ and fix $\Gamma = \langle M_{d;0,\alpha_{1},\alpha_{2}} \rangle \subset \GL(3,\kk)$, a finite cyclic group of order $d$ with $\alpha_1 < \alpha_{2}$. 
By \cite[Corollary 4.16]{CMM-R}, the homogeneous ideal $\I(X_d)$ of the $GT-$surface $X_d$ is minimally generated by binomials of degree $2$ and $3$ if $\GCD(\alpha_{1},d) = \GCD(\alpha_{2},d) = \GCD(\alpha_{2}-1,d) = 1$ and $\I(X_d)$ is minimally generated by  binomials of degree $2$ otherwise.

(ii) Fix integers $3 = n < d$ and fix $\Gamma = \langle M_{d;0,1,2,3} \rangle \subset \GL(4,\kk)$. In \cite[Corollary 5.7]{CM-R}, it is determined a minimal set of homogeneous binomials generators of the homogeneous ideal $\I(X_{d})$. Precisely,  it is proved that $\I(X_{d})$ is minimally generated by binomials of degree $2$ if $d$ is even and $\I(X_{d})$ is minimally generated by binomials of degree $2$ and $3$ if $d$ is odd. 

(iii) Fix integers $2 \leq n < d$ and fix  $\Gamma = \langle M_{d;0,\alpha_1, \hdots, \alpha_n} \rangle \subset \GL(n+1,\kk)$. If there exist integers $0 < \alpha_i < \alpha_j$ such that $\GCD(\alpha_{i},d) = \GCD(\alpha_j,d) = \GCD(\alpha_j-1,d) = 1$, then $\I(X_d)$ is minimally generated by binomials of degree $2$ and $3$. Indeed, let $V_{d}$ be the $GT-$surface associated to the action of the cyclic group $\Lambda := \langle M_{d;0,\alpha_i,\alpha_j} \rangle \subset \GL(3,\kk)$ of order $d$ acting on $\kk[x_0,x_i,x_j]$. Under a suitable identification of variables, we have that the homogeneous ideal $\I(V_{d}) = \I(X_{d}) \cap \kk[\rho^{-1}(x_0),\rho^{-1}(x_i), \rho^{-1}(x_j)]$. Therefore, if a minimal set of homogeneous binomial generators of $\I(V_{d})$ contains a binomial of degree $3$, the same holds for $\I(X_{d})$. 
\end{example}

\section{The canonical module of GT--varieties}
\label{Section: canonical module}

The algebraic structure of the canonical module $\omega_{X}$ of the homogeneous coordinate ring of an aCM projective variety $X$ plays a central role in its geometry (see \cite{Bruns-Herzog, Chardin, Kreuzer}). For example, it can lead us to derive information on the Hilbert function and series, as well as on the Castelnuovo-Mumford regularity, of the homogeneous coordinate ring of $X$. Let $\Gamma = \langle M_{d;\alpha_0,\hdots,\alpha_{n}} \rangle \subset \GL(n+1)$ be a cyclic group of order $d$ (see Notation \ref{Notation: group cyclic}). We denote by $I_d = (m_1,\hdots,m_{\mu_d}) \subset R$ the ideal generated by all monomial invariants of $\Gamma$ of degree $d$, $\varphi_{I_{d}}: \PP^{n} \to \PP^{\mu_d-1}$ the morphism induced by $I_d$ and $X_d = \varphi_{I_{d}}(\PP^{n}) \subset \PP^{\mu_d-1}$ its image. $X_d$ is an aCM projective variety and its homogeneous coordinate ring $S/\I(X_d)$ is isomorphic to $R^{\overline{\Gamma}}$ (see Theorem \ref{Theorem: Previos results CMM-R}). In this section, we deal with the canonical module $\omega_{X_{d}}$ of $S/\I(X_d)$. We identify $\omega_{X_d}$ with an ideal of $R^{\overline{\Gamma}}$ and we prove that it is generated by monomials of degree $d$ and $2d$. We focus on varieties $X_d$ the homogeneous coordinate rings $S/\I(X_d)$ of which are level ring, i.e. their canonical modules $\omega_{X_{d}}$ are generated in only one degree. Afterwards, we characterize the Castelnuovo-Mumford regularity of $R^{\overline{\Gamma}}$.

Let $m = x_0^{a_0}\cdots x_n^{a_n} \in R$ be a monomial, we say that $l_{m}:= (a_0,\hdots, a_n) \in \ZZ_{\geq 0}^{n+1}$ is {\em the lattice point associated to $m$}. Conversely, given a lattice point $l = (a_0,\hdots, a_n) \in \ZZ_{\geq 0}^{n+1}$, we say that $m_{l} := x_0^{a_0}\cdots x_n^{a_n}$ is {\em the monomial associated to $l$}. We denote $H_{d} \subset \ZZ_{\geq 0}^{n+1}$ the semigroup generated by the lattice points $l_{m_{1}}, \hdots, l_{m_{\mu_{d}}}$ associated to $\cM_{d}:=\{m_{1},\hdots, m_{\mu_{d}}\}$. By Theorem \ref{Theorem: Previos results CMM-R}(i), we have that $R^{\overline{\Gamma}} = \kk[H_{d}]$ and the semigroup $H_{d}$ coincides with the set of all solutions $(a_{0},\hdots,a_{n}) \in \ZZ_{\geq 0}^{n+1}$ of the systems:
\[(*)_{t,r} = \left\{\begin{array}{lclclclcl}
y_0 & + & y_{1} & + & \cdots &+& y_{n} &=& td\\
\alpha_0y_0 & + & \alpha_1y_{1} & + & \cdots &+& \alpha_ny_{n} &=& rd\\
\end{array}\right. \quad t \geq 1, \;\; r = 0,\hdots,\alpha_{n}t.\]
Therefore $H_{d}$ is a normal semigroup (see Section \ref{Section: preliminaries}) with $H_{d}^{+} := \{(a_0,\hdots,a_{n}) \in H_{d} \;\mid\; a_{i} > 0, \;i=0,\hdots,n\} \neq \emptyset$, for instance $(d,\hdots,d) \in H_{d}^{+}$. It holds that $H_{d}^{+}$ coincides with the relative interior of $H_{d}$. We set $\relint(I_{d}):= (x_0^{a_0}\cdots x_n^{a_n} \in R^{\overline{\Gamma}} \,\mid\, 0 \neq a_0\cdots a_n)$ the ideal
of $R^{\overline{\Gamma}}$ generated by all monomials associated to $H_{d}^{+}$. We have the following. 

\begin{proposition}\label{Proposition: Canonical module} $\relint(I_{d})$ is the canonical module of $R^{\overline{\Gamma}}$.
\end{proposition} 
\begin{proof} See \cite[Theorem 6.3.5(b)]{Bruns-Herzog}
\end{proof}
For a complete exposition of the canonical module of normal semigroup rings we refer the reader to \cite{Bruns-Herzog}.

We denote by $\mathcal{C}_{d,k} \subset \relint(I_{d})$ the set of all monomials of degree $kd$. We have:

\begin{theorem} \label{Theorem: Canonical Module Gt-variety}  For any cyclic group $\Gamma = \langle M_{d;\alpha_0,\hdots,\alpha_n} \rangle \subset \GL(n+1,\kk)$ of order $d$, $\relint(I_{d}) = \langle \mathcal{C}_{d,1}, \mathcal{C}_{d,2} \rangle$. 
\end{theorem}
\begin{proof} It is enough to show that for any monomial $m \in \mathcal{C}_{d,k}$, $k \geq 3$, there exists a monomial $m' \in \mathcal{C}_{d,k-1}$ which divides $m$. This proves that for $k \geq 3$, $\langle \mathcal{C}_{d,k} \rangle \subset \langle \mathcal{C}_{d,k-1} \rangle \subset \cdots \subset \langle \mathcal{C}_{d,2} \rangle$. 

We fix an integer $k \geq 3$, a monomial $m = x_{0}^{a_{0}}\!\cdots\! x_{n}^{a_{n}} \in \mathcal{C}_{d,k}$ and we set $m_{1} = m/(x_{0}\cdots x_{n}) = x_{0}^{a_0-1}\cdots x_{n}^{a_n-1}$. Since $d \geq n+1$ and $k \geq 3$, $m'$ is a monomial of degree $kd-(n+1) \geq 2d$. We define the sequence of integers $L = (\alpha_0, \overset{a_0-1}{\hdots}, \alpha_{0}, \hdots, \alpha_n, \overset{a_n-1}{\hdots}, \alpha_{n})$, by \cite[Theorem]{EGZ}  there exists a subsequence $L' = (\alpha_0, \overset{b_0}{\hdots}, \alpha_{0}, \hdots, \alpha_n, \overset{b_n}{\hdots}, \alpha_{n}) \subset L$ of $d$ integers the sum of which is a multiple of $d$. Therefore, $L'$ gives rise a monomial $m_{2} := x_{0}^{b_0}\cdots x_{n}^{b_n} \in R^{\overline{\Gamma}}$ of degree $d$ which divides $m_{1}$. Hence, we can factorize $m = m_{2}m'$ and by construction $m' \in \mathcal{C}_{d,k-1}$ is the required monomial. 
\end{proof}

\begin{example}\label{Example: Continuacion homo ideal and canonical module} \rm Let $\Gamma = \langle M_{6;0,1,2,3} \rangle \subset \GL(4,\kk)$ be a cyclic group of order $6$. We have that
$\mathcal{C}_{6,1} = \{x_0^3x_1x_2x_3, x_0x_1x_2x_3^3\}$  (see Example \ref{Example: 3fold homo ideal and canonical module}) and 

\[\begin{array}{lll}
\mathcal{C}_{6,2} & = & \{x_0^9 x_1 x_2x_3 , x_0^7 x_1 x_2x_3^3, x_0^6 x_1^2 x_2^2x_3^2, x_0^6 x_1 x_2^4x_3, x_0^5 x_1^4 x_2x_3^2, x_0^5 x_1^3 x_2^3x_3,x_0^5 x_1 x_2x_3^5, \\
& & x_0^4 x_1^5 x_2^2x_3 ,x_0^4 x_1^2 x_2^2x_3^4,x_0^4 x_1 x_2^4x_3^3, x_0^3 x_1^7 x_2x_3, x_0^3 x_1^4 x_2x_3^4 , x_0^3 x_1^3 x_2^3x_3^3,x_0^3 x_1^2 x_2^5x_3^2,  \\
& & x_0^3 x_1 x_2^7x_3, x_0^3 x_1 x_2x_3^7 ,x_0^2 x_1^5 x_2^2x_3^3, x_0^2 x_1^4 x_2^4x_3^2,x_0^2 x_1^3 x_2^6x_3, x_0^2 x_1^2 x_2^2x_3^6 ,x_0^2 x_1 x_2^4x_3^5,\\
& & x_0 x_1^7 x_2x_3^3, x_0 x_1^6 x_2^3x_3^2, x_0 x_1^5 x_2^5x_3, x_0 x_1^4 x_2x_3^6, x_0 x_1^3 x_2^3x_3^5,x_0 x_1^2 x_2^5x_3^4, x_0 x_1 x_2^7x_3^3 ,
\\
& & x_0 x_1 x_2 x_3^9 \}.
\end{array}\]
Only the following four monomials $x_0^2x_1^4x_2^4x_3^2, x_0^2x_1^3x_2^6x_3, x_0x_1^6x_2^3x_3^2,x_0x_1^5x_2^5x_3
 \in \mathcal{C}_{6,2}$ do not belong to the ideal $\langle \mathcal{C}_{6,1} \rangle \subset R^{\overline{\Gamma}}$. From this observation and Theorem \ref{Theorem: Canonical Module Gt-variety}, we obtain that the canonical module of $R^{\overline{\Gamma}}$ is the ideal: 
\[\relint(I_{6}) = (x_0^3x_1x_2x_3, x_0^2x_1^4x_2^4x_3^2, x_0^2x_1^3x_2^6x_3, x_0x_1^6x_2^3x_3^2, x_0x_1^5x_2^5x_3, x_0x_1x_2x_3^3).\]
In \cite{CM-R} a minimal free resolution of $R^{\overline{\Gamma}}$ is computed. Set $S := \kk[w_1,\hdots,w_{16}]$ (see Example \ref{Example: 3fold homo ideal and canonical module}), precisely we have: 

\[0\! \to\! S(-14)^{4} \oplus S(-15)^{2} \to S(-13)^{108} \oplus S(-14)^{7} \to S(-12)^{803} \to S(-11)^{2850} \to
S(-10)^{6237}\! \to \]
\[S(-9)^{9064} \to S(-7)^{6} \oplus S(-8)^{8811} \to
S(-6)^{258} \oplus S(-7)^{5352} \to S(-5)^{844} \oplus S(-6)^{1638} \to\]
\[S(-4)^{796} \oplus S(-5)^{184}\to S(-3)^{322}\oplus S(-4)^{13} \to S(-2)^{57} \to S \to S/\I(X_{6}) \to 0,\]
verifying as well Theorem \ref{Theorem: Canonical Module Gt-variety}. 
\end{example}

We recall that $R^{\overline{\Gamma}}$ is a {\em level} ring if its canonical module $\relint(I_{d})$ is generated in only one degree and $R^{\overline{\Gamma}}$ is a {\em Gorenstein} ring if it is a level ring and $\relint(I_{d})$ is principal. As a consequence of Theorem \ref{Theorem: Canonical Module Gt-variety}, we have that $R^{\overline{\Gamma}}$ is a level ring if and only if $\relint(I_{d}) = \langle \mathcal{C}_{d,1} \rangle$ or $\mathcal{C}_{d,1} = \emptyset$. 
In \cite[Corollary 4.13(ii)]{CMM-R}, it is shown that the homogeneous coordinate ring of any $GT-$surface is level. However, the same assertion is not true for $GT-$varieties of higher dimensions, see for instance \cite{CM-R}. We investigate further this property, which will play an important role in the last section of this work. Let us first present an interesting family of examples of Gorenstein $GT-$varieties. 

\begin{proposition} \label{Corollary: Goreinstein GT-variety} Fix an even integer $2 \leq n$  and $\Gamma = \langle M_{n+1;0,1,2,\hdots,n} \rangle \subset \GL(n+1,\kk)$, a cyclic group of order $n+1$. Then $R^{\overline{\Gamma}}$ is a Gorenstein ring.  
\end{proposition}
\begin{proof} Notice that $m = x_{0}\cdots x_{n} \in R^{\overline{\Gamma}}$. Indeed $m$ is of degree $n+1$ and it is satisfied that $1+2+\cdots+n = \frac{n(n+1)}{2}$, which is a multiple of $n+1$ since $n$ is even. It is straightforward to see that $m$ divides any monomial of $\relint(I_{n+1})$. Hence $\relint(I_{n+1}) = (x_{0}\cdots x_{n})$ and the proof is complete. 
\end{proof}

\begin{proposition}\label{Proposition: level from Gorenstein} Fix integers $2 \leq n < d$, $1 \leq k $ and assume that $n$ is even. Let $\Gamma = \langle M_{d;\alpha_0,\hdots, \alpha_n} \rangle \subset \GL(n+1,\kk)$ and $\Gamma_{k} = \langle M_{kd;\alpha_0,\hdots, \alpha_n} \rangle \subset \GL(n+1,\kk)$ be finite cyclic groups of order $d$ and $kd$, respectively. If $R^{\overline{\Gamma}}$ is a Gorenstein ring, then $R^{\overline{\Gamma_{k}}}$ is a level ring. 
\end{proposition}

\begin{proof} We denote $I_{d} \subset R$ (respectively $I_{kd}\subset R$) the ideal generated by all monomials of degree $d$ (respectively $kd)$ which are invariants of $\Gamma$ (respectively $\Gamma_{k})$. We write $\relint(I_{d}) = \langle m \rangle$. We want to prove that any monomial $m' \in \mathcal{C}_{d,2kd}$ is divisible by a monomial $\bar{m} \in \mathcal{C}_{d,kd}$. We fix $m' = x_0^{a_0}\cdots x_{n}^{a_{n}} \in \mathcal{C}_{d,2kd}$. Notice that $m'$ is also an invariant of $\Gamma$, so $m' \in \mathcal{C}_{d,k}$ and by hypothesis $m$ divides $m'$. We define $m_{1} = \frac{m'}{m}$; since $m_{1} \in  R^{\Gamma}$ is a monomial of degree $(2k-1)d$, by Theorem \ref{Theorem: Previos results CMM-R}(i),  there are monomials $m_{2},\hdots, m_{2k} \in R^{\Gamma}$ of degree $d$ such that $m_{1} = m_{2}\cdots m_{2k}$, and hence $m' = mm_{2}\cdots m_{2k}$.
For each monomial $m_{i}$, $1 \leq i \leq 2k$, there is a unique integer $r_{i} \geq 0$ such that the lattice point $l_{m_{i}}$ is a solution of the system $(*)_{1,r_{i}}$ induced by $\overline{\Gamma}$. 
By Lemma \ref{Lemma: EGZ}, there is a subsequence $(r_{i_{1}},\hdots, r_{i_{k}}) \subset (r_{1},\hdots, r_{2k})$ of $k$ integers the sum of which is a multiple of $k$. Therefore, we obtain that $m_{i_{1}}\cdots m_{i_{k}} \in R^{\overline{\Gamma_{k}}}$ and $\overline{m} = m'/(m_{i_{1}}\cdots m_{i_{k}}) \in \relint(I_{kd})$ is the required monomial. 
\end{proof}

\begin{corollary}\label{Corollary: level M(123...n)} Fix integers $2 \leq n$, $1 \leq k$ with $n$ even and fix $\Gamma_{k} = \langle M_{k(n+1);0,1,2,\hdots,n} \rangle \subset \GL(n+1,\kk)$, a finite cyclic group of order $k(n+1)$. Then $R^{\overline{\Gamma_{k}}}$ is a level ring. 
\end{corollary}
\begin{proof} It follows directly from Propositions \ref{Corollary: Goreinstein GT-variety} and \ref{Proposition: level from Gorenstein}.
\end{proof}

The rest of this section concerns the Castelnuovo-Mumford regularity $\reg(R^{\overline{\Gamma}})$ of $R^{\overline{\Gamma}}$ with $\Gamma = \langle M_{d;\alpha_0,\hdots,\alpha_n} \rangle \subset \GL(n+1,\kk)$ a cyclic group of order $d$ as in Notation \ref{Notation: group cyclic}. We characterize $\reg(R^{\overline{\Gamma}})$ in terms of $\relint(I_{d})$.   
First we need some preparation. 

\begin{definition} \em A set $\{y_{0},\hdots, y_{n}\} \subset R^{\overline{\Gamma}}$ of $n$ homogeneous elements is said to be a {\em homogeneous system of parameters}, shortly h.s.o.p, if $R^{\overline{\Gamma}}$ is a finitely generated $\kk[y_0,\hdots, y_n]$-module.  
\end{definition}

For sake of completeness we prove the following. 
\begin{proposition}\label{Proposition: h.s.o.p} $\{x_{0}^{d},\hdots, x_{n}^{d}\}$ is an h.s.o.p of 
$R^{\overline{\Gamma}}$.
\end{proposition}
\begin{proof} We consider the graded quotient algebra $A := R^{\overline{\Gamma}}/\langle x_0^{d},\hdots, x_{n}^{d} \rangle = \bigoplus_{t \geq 1} A_{t},$\; $A_{t} = R^{\overline{\Gamma}}_{t}/\langle x_0^{d},\hdots, x_{n}^{d} \rangle_{td} = R^{\Gamma}_{td}/\langle x_0^{d},\hdots, x_{n}^{d} \rangle_{td}$. Therefore, for $t \geq n+1$ we have that $A_{t} = \langle 0 \rangle$ and for $1 \leq t \leq n$, a $\kk-$basis of $A_{t}$ is formed by the set of all monomials $m = x_0^{a_0}\cdots x_n^{a_n}\in R^{\Gamma}$ of degree $td$ such that $a_0 < d, \hdots, a_n < d$. We write $\theta_{1}, \hdots, \theta_{D}$ the set of all such monomials and $\theta_{0} = 1$. Then, it is clear that $R^{\overline{\Gamma}} = \langle \theta_{0}, \theta_{1},\hdots, \theta_{D} \rangle$ as a  $\kk[x_0^{d},\hdots, x_{n}^{d}]$-module. 
\end{proof}

We call $\{\theta_{0},\theta_{1},\hdots, \theta_{D}\}$ a set of {\em secondary invariants} of $\overline{\Gamma}$.  For $1 \leq i \leq D$, we denote $\delta_{i} = deg(\theta_{i})$. We set $e_{0} = 1$ and we define $e_{j}$, $j = 1,\hdots, n$, the multiplicities of the sequence of degrees $(\delta_{1},\delta_{1},\hdots, \delta_{D})$. Notice that $e_{1} = \mu_{d}-(n+1)$. Moreover, we have the following. 

\begin{proposition}\label{Proposition:degree and Hilbert series} (i) The number of secondary invariants of $R^{\overline{\Gamma}}$ is $D+1 = d^{n-1}$. 

(ii) The Hilbert series $\HS(R^{\overline{\Gamma}},z)$ of the ring $R^{\overline{\Gamma}}$ is 
\begin{equation}\label{Equation:Hilbertseries}
\HS(R^{\overline{\Gamma}},z) = \frac{1 + e_1z + \cdots +e_{n}z^{n}}{(1-z)^{n+1}}.
\end{equation}
In particular, the degree of $X_{d}$ is $d^{n-1}$. 
\end{proposition}
\begin{proof}
See \cite[Proposition 2.3.6]{Sturmfelds}.
\end{proof}

\begin{theorem}\label{Theorem: regularity} With the above notation,
\[n \leq \reg(R^{\overline{\Gamma}}) \leq n+1.\]
The equality $\reg(R^{\overline{\Gamma}}) = n+1$ holds if and only if $\mathcal{C}_{d,1} \neq \emptyset$. 
\end{theorem}
\begin{proof} The right inequality follows immediately from (\ref{Equation:Hilbertseries}). We set $m = x_0^{d}\cdots x_n^{d}$ and $m' = x_{0}^{d-1}\cdots x_n^{d-1}$. Lemma \ref{Lemma: EGZ} assures the existence of a monomial of degree $(n-1)d$ in $R^{\overline{\Gamma}}$ dividing $m'$, and hence it assures the existence of secondary invariants of degrees smaller or equal to $(n-1)d$, so $e_{n-1} > 0$ which gives us the left inequality. Now, $\reg(R^{\overline{\Gamma}}) = n+1$ if and only if $e_{n} > 0$. If $e_{n} > 0$, then there exists a secondary invariant $\theta$ of degree $nd$ and we obtain $m/\theta \in \mathcal{C}_{d,1}$. Conversely, let $p = x_{0}^{a_0}\cdots x_{n}^{a_{n}} \in \mathcal{C}_{1,d}$. Notice that necessarily $a_{i} < d$, $i = 0,\hdots, n$, thus $m/p$ is a secondary invariant of degree $nd$. 
\end{proof}

To end this section,  we present some examples illustrating the last results. They also bring to light how the Hilbert series and regularity of $R^{\overline{\Gamma}}$ can be deduced by just looking at the set of invariants of $\overline{\Gamma}$ of degree smaller or equal to $nd$, and vice versa.

\begin{example}\rm (i) Fix integers $2 = n < d$ and fix $\Gamma = \langle M_{d;0,\alpha_{1},\alpha_{2}} \rangle \subset \GL(3,\kk)$ with $0 < \alpha_{1} < \alpha_{2}$. Let $\lambda, 0 < \mu \leq 
\frac{d}{\GCD(a,d)}$ be the uniquely determined integers satisfying $\alpha_{2} = \lambda \frac{\alpha_{1}}{\GCD(\alpha_{1},d)} + \mu \frac{d}{\GCD(\alpha_{1},d)}$. In \cite[Proposition 4.12]{CMM-R}, it is proved that 
\[\HS(R^{\overline{\Gamma}},z) = \frac{\frac{d-\theta(\alpha_{1},\alpha_{2},d)+2}{2}z^{2} + \frac{d+\theta(\alpha_{1},\alpha_{2},d)-4}{2}z + 1}{(1-z)^{3}},\]
where $\theta(\alpha_{1},\alpha_{2},d) = \GCD(\alpha_{1},d) + \GCD(\lambda,d') + \GCD(\lambda-1,d')$. 
By Proposition \ref{Proposition:degree and Hilbert series}, there are exactly $\frac{d+\theta(\alpha_{1},\alpha_{2},d)-4}{2}$ secondary invariants of degree $d$ and 
$\frac{d-\theta(\alpha_{1},\alpha_{2},d)+2}{2}$ secondary invariants of degree $2d$. 
From this it can be easily deduced that $\reg(R^{\overline{\Gamma}}) = 3$. 

(ii) Fix $n = 3$, $d = 4$ and fix $\Gamma = \langle M_{4,0,1,2,3} \rangle \subset \GL(4,\kk)$. We denote $I_{4}$ its associated $GT-$system. The ideal $I_{4}$ is generated by the following $10$ monomials
 \[x_0^4,x_0^2 x_2^2, x_0^2 x_1x_3,x_0 x_1^2 x_2, x_0 x_1x_3^2,x_1^4, x_1^2 x_3^2, x_1 x_2^2x_3, x_2^4,  x_3^4.\]
Therefore $\relint(I_{4}) = \langle \mathcal{C}_{4,2} \rangle$ with $\mathcal{C}_{4,2} = 
\{x_0^4 x_1 x_2^2x_3, x_0^3 x_1^3 x_2x_3 , x_0^3 x_1 x_2x_3^3 , x_0 x_1^3 x_2^3x_3,x_0^2 x_1^2 x_2^2x_3^2,$ $x_0^2 x_1 x_2^4x_3, x_0 x_1^4 x_2x_3^2, x_0 x_1^2 x_2x_3^4, x_0 x_1 x_2^3x_3^3  \}$. 
There are $6$ and $9$ secondary invariants of degree $d$ and $2d$, respectively, and by Theorem \ref{Theorem: regularity}, $\reg(R^{\overline{\Gamma}}) = 3$. By Proposition \ref{Proposition:degree and Hilbert series}, the Hilbert series of $R^{\overline{\Gamma}}$ is the following
\[\HS(R^{\overline{\Gamma}},z) = \frac{9z^2 + 6z+1}{(1-z)^{4}}.\]

(iii) Let $\Gamma = \langle M_{6,0,1,2,3} \rangle \subset \GL(4,\kk)$ be a cyclic group of order $6$. The ideal $I_{6}$ is generated by $16$ monomial invariants of $\Gamma$ of degree $6$ and the ideal $\relint(I_{6}) = (x_0^3x_1x_2x_3, x_0^2x_1^4x_2^4x_3^2, $ $x_0^2x_1^3x_2^6x_3, x_0x_1^6x_2^3x_3^2, x_0x_1^5x_2^5x_3, x_0x_1x_2x_3^3)$ (see Examples \ref{Example: 3fold homo ideal and canonical module} and \ref{Example: Continuacion homo ideal and canonical module}). Therefore, $R^{\overline{\Gamma}}$ has $12$ secondary invariants of degree $d$, it has $2$ secondary invariants of degree $3d$ and by Theorem \ref{Theorem: regularity} or Example \ref{Example: Continuacion homo ideal and canonical module}, $\reg(R^{\overline{\Gamma}}) = 4$. Hence from Proposition \ref{Proposition:degree and Hilbert series}, we can deduce immediately that the Hilbert series of the ring $R^{\overline{\Gamma}}$ is the following
\[\HS(R^{\overline{\Gamma}},z) = \frac{2z^3 + 21z^2 + 12z + 1}{(1-z)^3}.\]

(iv) Fix integers $2 \leq n$ and $d = n+1$ with $n$ even, and fix $\Gamma = \langle M_{t(n+1);0,1,2,\hdots,n} \rangle \subset \GL(n+1,\kk)$. By Corollary \ref{Corollary: Goreinstein GT-variety}, $R^{\overline{\Gamma}}$ is Gorenstein and by Theorem \ref{Theorem: regularity}, we obtain  that $\reg(\overline{\Gamma}) = n+1$. 
\end{example}

\section{Cohomology of normal bundles of RL--varieties}
\label{Section: Cohomology of normal bundles of RL-varieties}

In this section, we introduce a new family of smooth rational monomial projections of the Veronese variety $\nu_{d}(\PP^{n}) \subset \PP^{\binom{n+d}{n}-1}$ which naturally arises from {\em level} $GT-$varieties (see Definition \ref{Definition: level GT-variety}) and we study  their canonical modules. We called them $RL-$varieties to stress the link with the notions of   the {\bf r}elative interior and  {\bf l}evelness.  We devote the rest of this work to determine the cohomology of the normal bundle of any $RL-$variety (see Theorem \ref{Theorem: dimension cohomology normal}). Both, the coordinate ring and the canonical module of $GT-$varieties, play an important role on our computations and the proof of Theorem \ref{Theorem: dimension cohomology normal} is inspired by \cite{Alzati-Re}. 

In Section \ref{Section: canonical module}, we have seen that the canonical module $\omega_{X_{d}}$ of a $GT-$variety $X_d$ with group a finite cyclic group $\Gamma = \langle M_{d;\alpha_0,\hdots, \alpha_n} \rangle \subset \GL(n+1,\kk)$ of order $d$ is identified with the ideal $\relint(I_d) = (x_0^{a_0}\cdots x_n^{a_n} \in R^{\overline{\Gamma}} \,\mid\, 0 \neq a_0 \cdots a_n)$ and we have proved that $\relint(I_d)$ is generated by monomials of degree $d$ and $2d$. We begin with the following definition.

\begin{definition}\label{Definition: level GT-variety} \rm Let $X_{d}$ be a $GT-$variety with group a finite cyclic group $\Gamma = \langle M_{d;\alpha_{0},\hdots, \alpha_{n}} \rangle \subset \GL(n+1,\kk)$ of order $d$ and associated $GT-$system $I_{d}$. We say that $X_{d}$ is a {\em level} $GT-$variety if $R^{\overline{\Gamma}}$ is a level ring and, in addition, $\reg(R^{\overline{\Gamma}}) = n+1$. Equivalently, if its canonical module $\relint(I_{d}) = \langle\mathcal{C}_{d,1} \rangle$. (See Proposition \ref{Proposition: Canonical module} and Theorem \ref{Theorem: Canonical Module Gt-variety}). 
\end{definition}

Let us see some examples of level $GT-$varieties of any dimension $2 \leq n$, and next a necessary condition on the matrices $M_{d;\alpha_0,\hdots, \alpha_{n}}$ for $X_{d}$ be level. 

\begin{example} \rm \label{Example: level GT-varieties} (i) All $GT-$surfaces are level (see \cite[Corollary 4.13]{CMM-R}). 

(ii) Fix integers $4 \leq n, 1 \leq k$ with $n$ even. For $d := k(n+1)$ and for the finite cyclic group $\Gamma = \langle M_{d;0,1,2,\hdots, n} \rangle \subset \GL(n+1,\kk)$ of order $d$, the associated $GT-$variety $X_{d}$ is level (see Corollary \ref{Corollary: level M(123...n)}). 

(iii) Fix integers $3 \leq n$, $1 \leq t$ with $n$ odd and fix a finite cyclic group $\Gamma_t = \langle M_{t(n+1);0,1,\hdots,1,2} \rangle$ of order $d = t(n+1)$. Then $X_{t(n+1)}$ is a level $GT-$variety. 

\begin{proof} We fix $t \geq 1$. The monomials $x_0^{a_0} \cdots x_n^{a_n}$ of degree $d$ in $R^{\overline{\Gamma}_1}$ are the $\ZZ_{\geq 0}^{n+1}$-solutions of the systems:
\[(*)_{r}\left\{\begin{array}{rclrclrclrclrcl}
y_0 &+& y_1 &+& \cdots &+& y_{n-1} &+& y_n & = & d\\
& & y_1 &+& \cdots &+& y_{n-1} &+& 2y_n & = & rd,
\end{array}\right. \quad \quad r = 0,1,2.\]
We consider new variables $\alpha = y_0, \beta = y_1+\cdots+y_{n-1}$ and $\gamma = y_n$ and the integer systems
\[(**)_{r} \left\{\begin{array}{rclrclrcl}
\alpha &+& \beta &+& \gamma &=& d\\
& & \beta &+& 2\gamma &=& rd,
\end{array}\right. \quad \quad r = 0,1,2.\]
Any solution $(a_0,\hdots, a_n) \in \ZZ_{\geq 0}^{n+1}$ of $(*)_{r}$ gives rise to a solution $(a_0,a_1+\cdots +a_{n-1},a_n) \in \ZZ_{\geq 0}^{3}$ of
$(**)_{r}$. Conversely, a solution $(\alpha,\beta,\gamma) \in \ZZ_{\geq 0}^{3}$ of $(**)_{r}$ gives rise to $\binom{\beta+n-2}{n-2}$ solutions $(\alpha, 
a_1,\hdots, a_{n-1},\gamma) \in \ZZ_{\geq 0}^{n+1}$ of $(**)_{r}$, where 
$a_{1}+\cdots + a_{n-1} = \beta$. The solutions $(\alpha,\beta,\gamma) \in \ZZ_{\geq 0}^{3}$ of one of the systems $(**)_{r}$ are $(d,0,0), (0,0,d)$ and $\{(\gamma,d-2\gamma,\gamma) \;\mid \; \gamma = 0,\hdots,\frac{d}{2}\}$.   Therefore we have
\[\mu_{d} = 2 + \sum_{\gamma=0}^{\frac{d}{2}} \binom{d-2\gamma + n-2}{n-2}.\]
Since $x_0 \cdots x_n$ is an invariant of $\Gamma_{1} \subset \GL(n+1,\kk)$, as in the proof of Proposition \ref{Corollary: Goreinstein GT-variety} it follows that $R^{\overline{\Gamma}_{1}}$ is Gorenstein; by Proposition \ref{Proposition: level from Gorenstein} we only have to check that $\mu_{d} \leq \binom{n-1+d}{n-1}$. For $n \geq 3$, it holds that 
\[\mu_{d} = 2 + \sum_{\substack{\gamma=n-1\\ \gamma \; \text{odd}}}^{d+n-2} \binom{\gamma}{n-2} \leq 2 + \sum_{\widetilde{\gamma}=n-2}^{d+n-2} \binom{\widetilde{\gamma}}{n-2} - (n-1) \leq \sum_{\widetilde{\gamma}=n-2}^{d+n-2} \binom{\widetilde{\gamma}}{n-2} = \binom{n-1+d}{n-1}.\]
\end{proof}
\end{example}

\begin{proposition}\label{Proposition: necessary condition to be level}
Fix integers $2 \leq n < d$, a finite cyclic group $\Gamma = \langle M_{d;\alpha_0,\hdots, \alpha_n} \rangle \subset \GL(n+1,\kk)$ of order $d$ and set $I_{d}$ the ideal generated by all monomial invariants of $\Gamma$ of degree $d$. If $\mathcal{C}_{d,1} \neq \emptyset$, then there are at least three indices two by two distinct. 
\end{proposition}
\begin{proof} By contradiction, we assume $M_{d;\alpha_{0}, \hdots, \alpha_{n}} = M_{d; 0,\stackrel{l+1}{\hdots}, 0, a, \stackrel{n-l}{\hdots}, a}$ with $0 < a < d$ such that $\GCD(a,d) = 1$. Therefore, for any monomial $m \in R^{\Gamma}$ of degree $d$
it holds that $\supp(m) \in \{x_{0},\hdots, x_{l}\}$ or $\supp(m) \in \{
x_{l+1},\hdots, x_{n}\}$. In other words, $\mathcal{C}_{d,1} = \emptyset$ which contradicts our hypothesis. 
\end{proof}

\begin{remark} \rm Let $\Gamma = \langle M_{d;\alpha_0,\hdots,\alpha_n} \rangle \subset \GL(n+1,\kk)$ be a finite cyclic group of order $d$ and $I_d$ the monomial artinian ideal generated by the set of all monomials $\{m_1,\hdots,m_{\mu_d}\} \subset R^{\Gamma}$ of degree $d$. Definition \ref{Definition: level GT-variety} can be extended to any variety $X_d = \varphi_{I_{d}}(\PP^{n}) \subset \PP^{\mu_d-1}$ even if it is not a $GT-$variety, or equivalently, $I_d$ is not a Togliatti system, i.e. the condition $\mu_d \leq \binom{n-1+d}{n-1}$ is not satisfied. Many notions and results of this section remain true in this more general setup. 
\end{remark}

From now onwards, we fix a level $GT-$variety $X_{d}$ with group a finite cyclic group  $\Gamma = \langle M_{d;\alpha_0,\hdots, \alpha_n} \rangle \subset \GL(n+1,\kk)$ of order $d$ and we denote $I_{d}$ its associated $GT-$system. We denote $\eta_{d} = |\mathcal{C}_{d,1}|$, i.e.  the number of monomials of degree $d$ in $\relint(I_{d})$ and we set $N_d:= \binom{n+d}{d} - \eta_{d}-1$. By $f_{d}: \PP^{n} \to \PP^{N_d}$ we denote the morphism induced by the inverse system $\relint(I_{d})^{-1}$. We denote $\cX_{d} = f_{d}(\PP^{n}) \subset \PP^{N}$ and we call $\cX_{d}$ the {\em $RL-$variety associated to $X_{d}$.} We have the following. 

\begin{proposition}\label{Proposition: Embedding} $\cX_{d}$ is a smooth rational variety and $f_{d}$ is an embedding. 
\end{proposition}

\begin{proof} $\cX_{d}$ is a toric variety parametrized by all monomials of degree $d$ in $\relint(I_{d})^{-1}$. It is straightforward to check that $\cX_{d}$ satisfies the smoothness criterion for toric varieties \cite[Chapter 5 - Corollary 3.2]{Gelfand-Kapranov-Zelevinsky}. In particular, $\relint(I_d)^{-1}$ contains all monomials $x_i^{d-1}x_j$ for all $i,j \in \{0,\hdots,n\}$, which is a sufficient condition for $f_{d}$ to be an embedding. 
\end{proof}

In \cite{Alzati-Re}, the authors develop a new method to compute the cohomology of the normal bundle of smooth rational projections of the Veronese variety $\nu_{n,d}(\PP^{n}) \subset \PP^{\binom{n+d}{n}-1}$ for which the parametrization is an embedding.  Let $\cX_{d}$ be an $RL-$variety associated to a level $GT-$variety $X_{d}$ with group $\Gamma$. By Proposition \ref{Proposition: Embedding}, any $RL-$variety $\cX_{d}$ is of this kind. Furthermore, the relation between $\cX_{d}$ and $\relint(I_d) \subset R^{\overline{\Gamma}}$ allows us to apply this approach to any $RL-$variety $\cX_{d}$. In this setting, we have the following presentation of the normal bundle $\cN_{\cX_{d}}$ of the $RL-$variety $\cX_{d} \subset \PP^{\N_{d}}$ (see \cite[(3.3)]{Alzati-Re}):
\begin{equation}\label{Equation: Exact sequence normal}
0 \to \mathcal{O}_{\PP^{n}}^{n+1}(1) \to \mathcal{O}_{\PP^{n}}^{N_{d}+1}(d) \to \cN_{\cX_{d}} \to 0.
\end{equation}

Taking the long sequence of cohomology for (\ref{Equation: Exact sequence normal}), we determine the cohomology $\kk-$vector spaces $\Comh^{i}(\cX_{d}, \cN_{\cX_{d}}(-k))$ in most cases, as the following result shows. 

\begin{proposition}\label{Proposition: H0-H1 normal bundle} Let $\cX_{d} \subset \PP^{N_{d}}$ be an $RL-$variety of dimension $n \geq 2$, we have:

(i) for all $0 < i < n-1$ and $k \in \ZZ$, $\Comh^{i}(\cX_{d},\cN_{\cX_{d}}(-k)) = 0$.

(ii) \[ \comh^{0}(\cX_{d},\cN_{\cX_{d}}(-k)) = \left\{\begin{array}{lcl}
(N_{d}+1)\binom{n+d-k}{n} - (n+1)\binom{n+1-k}{n} & \quad & k \leq 1\\
(N_{d}+1)\binom{n+d-k}{n}                           & \quad & 1 < k \leq d\\
0                                                 & \quad & \text{otherwise.}
\end{array}\right.\]

\end{proposition}
\begin{proof} We fix $k \in \ZZ$. We twist (\ref{Equation: Exact sequence normal}) by $\mathcal{O}_{\PP^{n}}(-k)$ and then we consider the long exact sequence of cohomology. For any $i$ and $k$ we obtain
\begin{equation}\label{Equation: long sequence cohomology} 
\to \Comh^{i}(\PP^{n},\mathcal{O}_{\PP^{n}}^{N_{d}+1}(d-k)) \to \Comh^i(\cX_{d},\cN_{\cX_{d}}(-k)) \to \Comh^{i+1}(\PP^{n},\mathcal{O}_{\PP^{n}}^{n+1}(1-k))\to
\end{equation}
From the additivity of the cohomology, it follows the vanishing $\Comh^{i}(\cX_{d},\cN_{\cX_{d}}(-k)) = 0$ for all $0 < i < n-1$. In addition, we obtain the presentation
\[0 \to \Comh^{0}(\PP^{n},\mathcal{O}_{\PP^{n}}^{n+1}(1-k)) \to 
\Comh^{0}(\PP^{n},\mathcal{O}_{\PP^{n}}^{N_{d}+1}(d-k)) \to \Comh^{0}(\cX_{d}, \cN_{\cX_{d}}(-k)) \to 0.\]
The result follows from the Bott formulas for the cohomology of $\PP^{n}$ (see \cite{Okonek-Schneider-Spindler}). 
\end{proof}

Thus far, we have determined the dimension of $\Comh^{i}(\cX_{d},\cN_{\cX_{d}}(-k))$ for any $i$ and $k$ except: 
\[\Comh^{n-1}(\cX_{d},\cN_{\cX_{d}}(-k)) \;\, and \,\; \Comh^{n}(\cX_{d},\cN_{\cX_{d}}(-k)).\]
To compute them, we apply Proposition \ref{Proposition: H0-H1 normal bundle}(i) to the long exact sequence of cohomology (\ref{Equation: long sequence cohomology}). For any $k$ we obtain the exact sequence
\begin{equation}\label{Equation:Hn-1-n normal}
0\!\to\! \Comh^{n-1}(\cX_{d},\cN_{\cX_{d}}(-k)) \to \Comh^{n}(\PP^{n},\mathcal{O}_{\PP^{n}}^{n+1}(1-k)) \to \Comh^{n}(\PP^{n}, \mathcal{O}_{\PP^{n}}^{N_{d}+1}(d-k)) \to \Comh^{n}(\cX_{d},\cN_{\cX_{d}}(-k))\!\to\! 0.
\end{equation}
As an immediate result, we get that for all $k < d+n+1$:
\begin{equation}
\Comh^{n}(\cX_{d},\cN_{\cX_{d}}(-k)) = 0 \quad \text{and} \quad \Comh^{n-1}(\cX_{d},\cN_{\cX_{d}}(-k)) \cong \Comh^{n}(\PP^{n},\mathcal{O}_{\PP^{n}}^{n+1}(1-k)).
\end{equation}
Thus, 
\begin{equation}
\comh^{n-1}(\cX_{d},\cN_{\cX_{d}}(-k)) = \left\{\begin{array}{lcl}
(n+1)\binom{k-2}{n} & \quad & n+2 \leq  k < d+n+1\\
0                       &  \quad & k \leq n+1.
\end{array}\right.
\end{equation}

We focus on computing $\Comh^{n-1}(\cX_{d},\cN_{\cX_{d}}(-k))$ for $k \geq d+n+1$. We need some preparation. By $\partial_{x_{0}},\hdots, \partial_{x_{n}}$ we denote the linear operators acting on $R$ as partial derivatives. Let $m \in R_{l}$ be a monomial and we write $m = x_0^{a_0}\cdots x_n^{a_{n}}$. We denote $\partial_{m}$ the composition of linear operators $\partial_{x_0}\stackrel{a_0}{\cdots}\partial_{x_{0}} \cdots \partial_{x_n}\stackrel{a_n}{\cdots}\partial_{x_n}.$

\begin{lemma}\label{Lemma: lema tecnic per H1} Let $m$ be a monomial of degree $k-d-n-1$, let $q$ and $q'$ be monomials of degree $k-n-1$ such that $m$ divides both $q$ and $q'$, and let $0 \leq i \neq j \leq n$ be integers. Then $x_{i} \partial_{m}q$ and $x_{j}\partial_{m}q'$ are linearly independent if and only if for any monomial $m' \neq m$ of degree $k-d-n-1$ which divides $q$ and $q'$, $x_{i}\partial_{m'}q$ and $x_j \partial_{m'}q'$ are linearly independent. 
\end{lemma}
\begin{proof}  We write $m' = x_{0}^{b_0}\cdots x_{n}^{b_{n}}$, $m = x_{0}^{c_0}\cdots x_{n}^{c_{n}}$, $q = x_{0}^{a_{0}}\cdots x_{n}^{a_{n}}$, $q' = x_{0}^{a_{0}'}\cdots x_{n}^{a_{n}'}$. Assume that $x_{i} \partial_{m}q$ and $x_{j}\partial_{m}q'$ are linearly independent and there is $m' \neq m$, which divides $q$ and $q'$, such that $x_{i}\partial_{m'}q$ and $x_j \partial_{m'}q'$ are linearly dependent. Therefore we have the equality
\[x_{0}^{a_{0}-b_0}\cdots x_{i}^{a_i-b_i+1} \cdots 
x_{n}^{a_n-b_n} = x_{0}^{a_0'-b_0} \cdots x_{j}^{a_j'-b_j+1}\cdots x_{n}^{a_n'-b_n},\]
which implies $a_{l} = a_{l}'$, $0 \leq l\neq i,j \leq n$, $a_i = 
a_i'-1$ and $a_j = a_j'+1.$ Then we obtain
\[\begin{array}{lll}
x_i\partial_{m}q & = & A x_{0}^{a_{0}-c_0}\cdots x_{j}^{a_j-c_j} \cdots 
x_{i}^{a_i-c_i+1}\cdots x_{n}^{a_n-c_n}\\
x_j\partial_{m}q'& = & B x_{0}^{a_0-c_0} \cdots x_{j}^{a_j-1-c_j+1}\cdots 
x_{i}^{a_i+1-c_i}\cdots x_n^{a_n-c_n}\end{array}\]
for some $A,B \in \kk\setminus\{0\}$, which is a contradiction. 
\end{proof}

An $RL-$variety $\cX_{d} \subset \PP^{N_d}$ of dimension $n \geq 2$ is a smooth rational variety embedded in $\PP^{N_d}$. In \cite{Alzati-Re}, the authors introduce a new method to compute the cohomology of the normal bundle of varieties of this kind. With the notation of \cite{Alzati-Re}, we write the embedding $f_{d}: \PP(U) \to \PP^{N_{d}}$ with $U = R_1^{\vee}$. The $RL-$variety $\cX_{d} = f_{d}(\PP(U))$ is the projection in $\PP^{N_{d}}$ of the Veronese variety  $\nu_{d}(\PP(U)) \subset \PP^{\binom{n+d}{n}-1}$ from the projective space $\PP(T)$ of dimension $\binom{n+d}{n}-N_{d}$ where $T^{\vee}$ is identified with the $\kk-$vector subspace of $R_{d}$ generated by all the monomials of degree $d$ in $\relint(I_{d}) = (x_0^{a_0}\cdots x_n^{a_n} \in R^{\overline{\Gamma}} \,\mid \, 0 \neq a_0\cdots a_n)$. Let $0 \leq i \neq j \leq n$, $l \geq 1$ and $t \geq 1$ be integers. We denote $D_{i,j}: S^{l}U \otimes S^{t}U \to S^{l-1}U \otimes S^{t-1}U$ the linear map $\partial_{x_i} \otimes \partial_{x_{j}} - \partial_{x_{j}} \otimes \partial_{x_{i}}$.

\begin{proposition}\label{Proposition: Hn-1 normal} Let $\cX_{d} \subset \PP^{N_{d}}$ be an $RL-$variety of dimension $n \geq 2$ associated to a level $GT-$variety $X_{d}$ with group $\Gamma = \langle M_{d;\alpha_0,\hdots,\alpha_n} \rangle \subset \GL(n+1,\kk)$. Then, 
\[\comh^{n-1}(\cX_{d},\cN_{\cX_{d}}(-k)) = \left\{\begin{array}{lcl}
\eta_{d} + \frac{n(d-1)}{d} \binom{n+d-1}{n} & \quad \quad & k = d+n+1\\
(n+1)\eta_{d} & \quad \quad & k = d+n+2\\
0 & \quad \quad & k \geq d+n+3.
\end{array}\right.\]
\end{proposition}
\begin{proof}  By \cite[Theorem 2]{Alzati-Re}, we obtain $\comh^{n-1}(\cX_{d},\cN_{\cX_{d}}(-d-n-1)) = \dim(\mu^{-1}(T))$, where $\mu: U \otimes S^{d-1}U \to S^dU$ is the multiplication map, and for all $k \geq d+n+2$: 
\[\Comh^{n-1} (\cX_{d},\cN_{\cX_{d}}(-k))\!=\!(S^{k-d-n-1}U \otimes T)  \cap (\bigcap_{0 \leq i,j,r,s \leq n}(ker(D_{i,j} \circ D_{r,s})).\]
In particular, for $k = d+n+1,d+n+2$ we can conclude that
\[\comh^{n-1}(\cX_{d},\cN_{\cX_{d}}(-d-n-1)) = \eta_{d} + \frac{n(d-1)}{d} \binom{n+d-1}{n}\] and $\Comh^{n-1}(\cX_{d},\cN_{\cX_{d}}(-d-n-2)) = U \otimes T.$ Moreover, for $k \geq d+n+3$ we have
\[\begin{array}{lr}
\Comh^{n-1}(\cX_{d},\cN_{\cX_{d}}(-k)) \cong \{x_0 \otimes q_{0} + \cdots 
+ x_n \otimes q_{n} \in R_{1} \otimes R_{k-n-2} \,\mid\,\\[0.2cm]
x_0\partial_{m}(q_0)\!+ \cdots +\! x_n \partial_{m}(q_{n}) \!\in\! \relint(I_d)_1 \,\text{for all monomial}\, m \!\in\! R_{k-d-n-1}\},
\end{array}\]
where $\relint(I_d)_1$ denotes the $\kk-$vector subspace of $R_{d}$ generated by the set $\mathcal{C}_{d,1} \subset \relint(I_d)$ of  monomials of degree $d$.  
We want to prove that $\Comh^{n-1}(\cX_{d},\cN_{\cX_{d}}(-k)) = 0$ for all $k \geq d+n+3$. 
Assume that there exist $q_0,\hdots,q_n \in R_{k-n-2}$ and a monomial $m \in R_{k-d-n-1}$ such that $0 \neq u_m := x_0\partial_{m}(q_0) + \cdots + x_n \partial_{m}(q_{n}) \in \relint(I_d)_1.$ Therefore, any monomial that appears in $u_m$ belongs to $\relint(I_d) \subset \overline{\Gamma}$. Let $q \in R_{k-n-2}$ be a monomial such that $0 \neq x_i \partial_{m}q$ is a monomial that occurs in $u_m$. By Lemma \ref{Lemma: lema tecnic per H1}, we have that if 
$x_0 \otimes q_0 + \cdots + x_n \otimes q_n \in \Comh^{n-1}(\cX_{d},\cN_{\cX_{d}}(-k)),$
then for any monomial $m' \in R_{k-d-n-1}$, \, $x_{i}\partial_{m'}q \in \relint(I_d)) \subset R^{\overline{\Gamma}}_{1} = R^{\Gamma}_{d}$. We will show that there always exists a monomial $m' \in R_{k-d-n-1}$ dividing $q$ such that $x_i\partial_{m'}q \notin R^{\overline{\Gamma}}$. Thus, it concludes $\Comh^{n-1}(\cX_{d},\cN_{\cX_{d}}(-k)) = 0$ for all $k \geq d+n+3$.
Notice that, by Proposition \ref{Proposition: necessary condition to be level}, $M_{d;\alpha_0,\hdots,\alpha_n}$ has three indices $\alpha_i,\alpha_j,\alpha_l$ two by two distinct. 
We consider monomials $q = x_0^{a_{0}}\cdots x_{n}^{a_{n}} \in R_{k-n-2} \; \text{and} \; m = x_{0}^{b_{0}}\cdots x_{n}^{b_{n}} \in R_{k-d-n-1}$ such that $m$ divides $q$ and $x_i\partial_{m}q \in \relint(I_d)$. In particular, we have that $b_{j} < a_{j}$ for all $0 \leq j \neq i \leq n$ and $b_i \leq a_i-1$. By assumption, $x_i \partial_{m}q := x_{0}^{c_{0}}\cdots x_{n}^{c_{n}} = x_{0}^{a_{0}-b_{0}}\cdots x_{i}^{a_{i}-b_{i}+1} \cdots x_{n}^{a_{n}-b_{n}} \in R^{\overline{\Gamma}}_1.$
We distinguish two cases. 

\vspace{0.15cm}
\noindent \underline{Case 1:} $0 < b_{i}$. If $\alpha_i = 0$ or $\alpha_i > 0$ and $2\alpha_i-\alpha_j  \not\equiv  0 \mod d$, we define the monomial $m' = x_{0}^{b_{0}}\cdots x_{l}^{b_{j}+1} \cdots x_{i}^{b_{i}-1} \cdots x_{n}^{b_{n}}.$ Then $x_i\partial_{m'}q = x_0^{c_0}\cdots x_{j}^{c_j-1} \cdots x_{i}^{c_i+1} \cdots x_{n}^{c_n}.$ Otherwise, $2\alpha_i - \alpha_l \not\equiv 0 \mod d$ and we define $m' = x_{0}^{b_{0}}\cdots x_{l}^{b_{l}+1} \cdots x_{i}^{b_{i}-1} \cdots x_{n}^{b_{n}}.$
Then $x_i\partial_{m'}q = x_0^{c_0}\cdots x_{l}^{c_l-1} \cdots x_{i}^{c_i+1} \cdots x_{n}^{c_n}$ and its associated point does not verify the linear congruence equation $\alpha_0y_0 + \cdots + \alpha_ny_n \equiv 0 \mod d.$

\vspace{0.15cm}
\noindent \underline{Case 2:} $b_i = 0$. We take $0 < b_{h}$, and we can assume that $\alpha_h,\alpha_j$ are different pair-wise. We define $m' = x_0^{b_0}\cdots x_j^{b_j+1} \cdots x_h^{b_h-1} \cdots x_n^{b_n}.$
Then, $x_i\partial_{m'}q = x_0^{c_0}\cdots x_{j}^{c_j-1} \cdots x_{i}^{c_i} \cdots x_{h}^{c_{h}+1} \cdots x_{n}^{c_n}$ and its associated point does not verify the linear congruence equation $\alpha_0y_0 + \cdots + \alpha_ny_n \equiv 0 \mod d.$

In any case, we have constructed a monomial $m' \in R_{k-d-n-1}$ dividing $q$ such that $x_i \partial_{m'}q \notin R^{\overline{\Gamma}}_{1}$ and the proposition follows. 
\end{proof}
Directly from (\ref{Equation:Hn-1-n normal}) and Proposition \ref{Proposition: Hn-1 normal}, we obtain $\comh^{n}(\cX_{d},\cN_{\cX_{d}}(-k))$ for all $k \geq d+n+1$. We summarize all the computations and establish the main result of this section. 

\begin{theorem}\label{Theorem: dimension cohomology normal} Fix a level $GT-$variety $X_{d}$ with group a finite cyclic group $\Gamma = \langle M_{d;\alpha_0,\hdots, \alpha_n} \rangle \subset \GL(n+1,\kk)$ of order $d$ and associated $GT-$system $I_{d}$. Set $\eta_{d} := |\mathcal{C}_{d,1}|$ the number of monomials of $\relint(I_{d})$ of degree $d$ and $N_{d} := 
\binom{n+d}{d}-\eta_{d}-1$. Let $\cX_{d} \subset \PP^{N_{d}}$ be the $RL-$variety of dimension $n \geq 2$ associated to $X_{d}$. It holds: 

\noindent (i) for $0 < i < n-1$ and for all $k \in \ZZ$, \quad
$\comh^{i}(\cX_d, \cN_{\cX_{d}}(-k)) = 0$.

\vspace{0.15cm}
\noindent (ii) \[ \comh^{0}(\cX_{d},\cN_{\cX_{d}}(-k)) = \left\{\begin{array}{lcl}
(N_{d}+1)\binom{n+d-k}{n} - (n+1)\binom{n+1-k}{n} & \quad & k \leq 1 \\
(N_{d}+1)\binom{n+d-k}{n}                           & \quad & 1 < k \leq d\\
0                                                 & \quad & \text{otherwise.}
\end{array}\right.\]
\noindent (iii) \[\comh^{n-1}(\cX_{d},\cN_{\cX_{d}}(-k)) = \left\{\begin{array}{lcl}
(n+1)\binom{k-2}{n} & \quad & n+2 \leq  k < d+n+1\\
\eta_{d} + \frac{n(d-1)}{d} \binom{n+d-1}{n} & \quad  & k = d+n+1\\
(n+1)\eta_{d} & \quad & k = d+n+2\\
0                       &  \quad & k \leq n+1 \;\; \text{or} \;\;  k \geq d+n+3.
\end{array}\right.\]
\noindent (iv) \[\comh^{n}(\cX_{d},\cN_{\cX_{d}}(-k)) = \left\{\begin{array}{l@{}c@{}l}
(N_{d}+1)\binom{k-d-1}{n} - (n+1)\binom{k-2}{n} & & k \geq d+n+3\\
0 & \quad & \text{otherwise.} 
\end{array}\right.\]
\end{theorem}

We end this work by showing two examples pointing out Theorem \ref{Theorem: dimension cohomology normal}. All the computations have been made with the software Macaulay2 (\cite{Macaulay2}). 

\begin{example}\rm (i) We fix $d = 5$ and $\Gamma = \langle M_{5;0,1,2} \rangle \subset \GL(3,\kk)$ a cyclic group of order $5$. The ideal $I_{5} = (x_0^5,x_0^2 x_1 x_2^2,x_0 x_1^3 x_2,x_1^5,x_2^5) \subset \kk[x_0,x_1,x_2]$ is the $GT-$system generated by all monomial invariants of $\Gamma$ of degree $5$. Its associated $GT-$variety $X_{5}$ is level with $\relint(I_{5}) = (x_0^2x_1x_2^2, x_0x_1^3x_2)$ and $\eta_{5} = 2$. We present the cohomology table from degree $-10$ to $0$ of the normal bundle $\cN_{\cX_{5}}$ of the smooth rational variety $\cX_{5}$ parametrized by the inverse system $\relint(I_{5})^{-1}$. 

\[\begin{array}{lrrrrrrrrrrrrrrrrrrrrrrr}
 
   & -10 & -9  & -8 & -7 & -6 & -5 & -4 & -3 & -2 & -1 & 0 \\
2: & 150 & 82  & 30 &  . &  . &  . &  . &  . &  . &  . & . \\
1: & .   & .   & 6  & 26 & 30 & 18 & 9  & 3  &  . &  . & .  \\
0: & .   & .   & .  &  . &  . & 19 & 57 & 114&190 & 282& 390 \\

\end{array}\]

\vspace{0.3cm}
(ii) We fix $d = 4$ and $\Gamma = \langle M_{4;0,1,1,2} \rangle \subset \GL(4,\kk)$ a cyclic group of order $4$. The associated $GT-$system $I_{4} = (x_0^4,x_0^2x_3^2, x_0 x_1^2x_3, x_0 x_1 x_2x_3,x_0 x_2^2x_3,x_1^4,x_1^3 x_2,x_1^2 x_2^2,x_1 x_2^3,x_2^4,x_3^4) \subset \kk[x_0,x_1,x_2,x_3]$ and its associated $GT-$threefold $X_{4}$ is level with 
$\relint(I_{4}) = (x_0x_1x_2x_3)$ (see Example \ref{Example: level GT-varieties}(iii)). We present the cohomology table from degree $-9$ to $0$ of the normal bundle $\cN_{\cX_{4}}$ of the smooth rational variety $\cX_{4}$ parametrized by the inverse system $\relint(I_{4})^{-1}$. 

\[\begin{array}{lrrrrrrrrrrrrrrrrrrrrrr}
 
   &  -9 & -8 & -7 & -6 & -5 & -4 & -3 & -2 & -1 &  0 \\
3: & 710 &344 &116 &  . & .  &  . &  . &  . &  . &  . \\
2: & .   & .  & 4  & 46 & 40 & 16 &  4 &  . &  . &  . \\
1: & .   & .  & .  & .  & .  & .  & .  &  . &  . &  . \\
0: & .   & .  & .  &  . &  . &34  & 136&340 &676 & 1174 \\

\end{array}\]    
      
\end{example}

%%%%%%%%%%%%%%%%%%%%%%%%%%%%%%%%%%%%%%
%%%%%%%%%%%%%%%%%%%%%%%%%%%%%%%%%%%%%%%%%%%%%%%%%%

\end{document}